%% file: arxiv.tex
\documentclass[preprint,12pt]{imsart}

\RequirePackage{amsthm,amsmath,amsfonts,amssymb}
\RequirePackage[authoryear]{natbib}
\RequirePackage{graphicx}
\RequirePackage{fullpage}
\RequirePackage[utf8]{inputenc}
\RequirePackage{booktabs,rotating}
\RequirePackage{grffile}
\RequirePackage{bm}
\RequirePackage{enumerate}

\startlocaldefs
\numberwithin{equation}{section}
\theoremstyle{plain}
\newtheorem{prop}{Proposition}[section]

\newtheorem{thm}[prop]{Theorem}
\newtheorem{cor}[prop]{Corollary}
\newtheorem{lem}[prop]{Lemma}
\newtheorem{cond}[prop]{Condition}
\theoremstyle{remark}
\newtheorem{remark}[prop]{Remark}

\newcommand{\eps}{\varepsilon}
\newcommand{\N}{\mathbb{N}}
\newcommand{\R}{\mathbb{R}}
\newcommand{\Z}{\mathbb{Z}}
\newcommand{\dd}{\mathrm{d}}

\def\disp{\displaystyle}

\newcommand{\Ex}{\mathbb{E}}

\newcommand{\Cov}{\mathrm{Cov}}
\newcommand{\1}{\mathbf{1}}
\newcommand{\ip}[1]{\lfloor #1 \rfloor}
\renewcommand{\Pr}{\mathbb{P}}
\newcommand{\p}{\overset{\Pr}{\to}}
\newcommand{\as}{\overset{as}{\to}}

\newcommand{\C}{\mathrm{C}}
\newcommand{\IF}{{I\hspace{-0.7mm}F}}
\newcommand{\sss}[1]{\scriptscriptstyle (#1)}
\newcommand{\sm}{\scriptscriptstyle (m)}
\newcommand{\sz}{\scriptscriptstyle (0)}

\endlocaldefs

\begin{document}

\begin{frontmatter}
\title{Open-end nonparametric sequential change-point detection based on the retrospective CUSUM statistic}
\runtitle{Open-end monitoring based on the retrospective CUSUM statistic}

\begin{aug}
  \author[M]{\fnms{Mark} \snm{Holmes}\ead[label=eM]{holmes.m@unimelb.edu.au}}
  \and
  \author[I]{\fnms{Ivan} \snm{Kojadinovic}\ead[label=eI]{ivan.kojadinovic@univ-pau.fr}}
  \address[M]{School of Mathematics \& Statistics, The University of Melbourne, Parkville, VIC 3010, Australia, \printead{eM}}
  \address[I]{CNRS / Universit\'e de Pau et des Pays de l'Adour / E2S UPPA, Laboratoire de math\'ematiques et applications, IPRA, UMR 5142, B.P. 1155, 64013 Pau Cedex, France, \printead{eI}}
\end{aug}

\begin{abstract}
The aim of online monitoring is to issue an alarm as soon as there is significant evidence in the collected observations to suggest that the underlying data generating mechanism has changed. This work is concerned with open-end, nonparametric procedures that can be interpreted as statistical tests. The proposed monitoring schemes consist of computing the so-called retrospective CUSUM statistic (or minor variations thereof) after the arrival of each new observation. After proposing suitable threshold functions for the chosen detectors, the asymptotic validity of the procedures is investigated in the special case of monitoring for changes in the mean, both under the null hypothesis of stationarity and relevant alternatives. To carry out the sequential tests in practice, an approach based on an asymptotic regression model is used to estimate high quantiles of relevant limiting distributions. Monte Carlo experiments demonstrate the good finite-sample behavior of the proposed monitoring schemes and suggest that they are superior to existing competitors as long as changes do not occur at the very beginning of the monitoring. Extensions to statistics exhibiting an asymptotic mean-like behavior are briefly discussed. Finally, the application of the derived sequential change-point detection tests is  succinctly illustrated on temperature anomaly data.
\end{abstract}



\end{frontmatter}


\section{Introduction}

In situations in which observations are continuously collected over time, the aim of \emph{sequential} or \emph{online change-point detection} is to issue an alarm as soon as possible if it is thought that the probabilistic properties of the underlying unobservable data generating mechanism have changed. While this problem has a long history in \emph{statistical process control} \citep[see, e.g.,][for an overview]{Lai01,Mon07}, we adopt herein the alternative perspective proposed in the seminal work of \cite{ChuStiWhi96} and treat the issue from the point of view of statistical testing. To fix ideas, assume that we have at hand an initial stretch $X_1,\dots,X_m$ (frequently referred to as the \emph{learning sample}) from a univariate stationary time series $(X_i)_{i \in \Z}$. As new observations arrive, the aim is to look for evidence against stationarity and issue an alarm if such evidence is deemed significant. Specifically, assuming that the $k$th observation has been collected, a positive statistic $D_m(k)$ measuring departure from stationarity is computed from the available observations $X_1,\dots,X_k$ and compared to a suitably chosen threshold $w(k/m)$. If $D_m(k) > w(k/m)$, the hypothesis that $X_1,\dots,X_k$ is a stretch from a stationary time series is rejected and the monitoring stops. Otherwise, a new observation $X_{k+1}$ is collected and the previous steps are repeated using $X_1,\dots,X_{k+1}$.

Among such procedures, one can distinguish between \emph{closed-end} approaches for which the monitoring eventually stops if stationarity is not rejected after the arrival of observation $X_n$, $n > m$, and \emph{open-end} approaches which, in principle, can continue indefinitely. The sequential testing procedures studied in this work pertain to this second category and, for that reason, their null hypothesis is
\begin{equation}
  \label{eq:H0}
  H_0: \,  X_1,\dots, X_m, X_{m+1},X_{m+2},\dots, \text{ is a stretch from a stationary time series}.
\end{equation}

Their interpretation as statistical tests implies that, given a significance level $\alpha \in (0,1/2)$, the thresholds $w(k/m)$, $k \geq m+1$, need to be chosen such that, ideally, under $H_0$,
\begin{equation}
  \label{eq:typeIerror}
  \Pr\{D_m(k) \leq w(k/m) \text{ for all } k \geq m+1 \} =  \Pr \left\{ \sup_{m+1 \leq k < \infty} \frac{D_m(k)}{w(k/m)} \leq 1 \right\} \geq 1 - \alpha.
\end{equation}
The previous display shows how closely the choice of the so-called \emph{detector} $D_m$ and the \emph{threshold function} $w$ are related. Intuitively, under $H_0$, the sample paths of the stochastic process $\{ D_m(k)/w(k/m) \}_{k \geq m+1}$ should fluctuate but not exceed a constant threshold too often while, when $H_0$ is not true, they are expected to rapidly exceed it after a change in the data generating process has occurred.

The starting point of our investigations is the recent seminal work of \cite{GosKleDet20} who study open-end monitoring schemes sensitive to potential changes in a parameter $\theta$ (such as the mean) of a time series. A first natural approach to tackle this problem, often referred to as the \emph{ordinary CUSUM (cumulative sum)} in the sequential change-point detection literature, consists of comparing an estimator $\theta_{1:m}$ of $\theta$ computed from the learning sample $X_1,\dots,X_m$ with an estimator $\theta_{m+1:k}$ of $\theta$ computed from the observations $X_{m+1},\dots,X_k$ collected after the monitoring has started; see, e.g., \cite{HorHusKokSte04}, \cite{AueHorHusKok06} as well as the references given in the recent review by \cite{KirWeb18}. The idea of \cite{GosKleDet20} is to define detectors that take into account \emph{all} of the differences $\theta_{1:j} - \theta_{j+1:k}$, $j \in \{m+1,\dots,k-1\}$. This approach, that can be regarded as adapted from \emph{retrospective} (or \emph{offline}, \emph{a posteriori}) change-point detection (which assumes that data collection has been completed \emph{before} testing is carried out), treats each $j \in \{m+1,\dots,k-1\}$ as a potential change-point. In their experiments, \cite{GosKleDet20} found that their approach is not only more powerful than the ordinary CUSUM but also more powerful than the so-called \emph{Page CUSUM} procedure which consists of defining a detector from the differences $\theta_{1:m} - \theta_{j+1:k}$, $j \in \{m+1,\dots,k-1\}$; see, e.g, \cite{Fre15} and \cite{KirWeb18}.

When computed after the $k$th observation has been collected, the detector proposed by \cite{GosKleDet20} (thus involving all the differences $\theta_{1:j} - \theta_{j+1:k}$, $j \in \{m+1,\dots,k-1\}$) does not however coincide with the so-called \emph{retrospective CUSUM statistic} that could be computed from $X_1,\dots,X_k$ (also involving all the differences $\theta_{1:j} - \theta_{j+1:k}$, $j \in \{m+1,\dots,k-1\}$;  see, e.g., \citealp{CsoHor97,AueHor13} and the references therein). Reasons that led \cite{GosKleDet20} not to consider such an approach are discussed in their Remark~2.1. As shall be explained in the next section, we believe that an open-end monitoring scheme using the retrospective CUSUM statistic as detector could be even more powerful than the procedure of \cite{GosKleDet20} as long as changes do not occur at the very beginning of the monitoring.

The aim of this work is to address the theoretical and practical issues associated with defining a nonparametric detector for open-end monitoring such that it coincides at each~$k$ with the retrospective CUSUM statistic. The theoretical issues are mostly related to the choice of the threshold function, while the practical issues come from the fact that quantiles of the underlying limiting distribution required to carry out the sequential test are harder to estimate.

This paper is organized as follows. In the second section, we propose three open-end nonparametric monitoring schemes related to the retrospective CUSUM statistic designed to be sensitive to changes in the mean of univariate time series. Their asymptotic behavior as the size $m$ of the learning sample tends to infinity is studied in the third section both under the null hypothesis of stationarity and relevant alternatives. Section~\ref{sec:quantiles} is concerned with the estimation of high quantiles of related limiting distributions necessary in practice to carry out the sequential tests. The fifth section presents a summary of extensive numerical experiments demonstrating the good finite-sample properties of the resulting sequential testing procedures. An extension to time series parameters whose estimators admit an asymptotic mean-like linearization as considered in \cite{GosKleDet20} is briefly discussed in Section~\ref{sec:linear}. A short illustration involving temperature anomaly data concludes the work.

All proofs are deferred to the appendices. Throughout the paper, all convergences are for $m \to \infty$ unless mentioned otherwise. A preliminary implementation of the studied tests is available in the package {\tt npcp} \citep{npcp} for the \textsf{R} statistical system \citep{Rsystem}.

\section{The retrospective CUSUM for monitoring changes in the mean}
\label{sec:det}

Our aim is to derive open-end nonparametric sequential change-point detection procedures that are particularly sensitive to alternative hypotheses of the form
\begin{equation}
  \label{eq:H1}
  H_1: \, \exists \, k^\star \geq m \text{ such that }  \Ex(X_1) = \dots = \Ex(X_{k^\star}) \neq \Ex(X_{k^\star+1}) = \Ex(X_{k^\star+2}) = \dots.
\end{equation}

After the arrival of the $k$th observation with $k > m$, the data at hand consist of the stretch $X_1,\dots,X_k$. If we were in the context of retrospective change-point detection, a natural test statistic would be the so-called \emph{retrospective CUSUM statistic} \citep[see, e.g.,][and the references therein]{CsoHor97,AueHor13} defined by
\begin{equation}
  \label{eq:Rk}
R_k = \max_{1 \leq j \leq k-1}  \frac{j (k - j)}{k^{3/2}} |  \bar X_{1:j} - \bar X_{j+1:k} |,
\end{equation}
where
\begin{equation}
  \label{eq:Xjk}
   \bar X_{j:k} =
  \left\{
    \begin{array}{ll}
      \disp \frac{1}{k-j+1}\sum_{i=j}^k X_i, \qquad & \text{if } j \leq k, \\
      0, \qquad &\text{otherwise}.
    \end{array}
  \right.
\end{equation}
In the definition of $R_k$, we see that every $j \in \{1,\dots,k-1\}$ is treated as a potential change-point in the sequence $X_1,\dots,X_k$. The maximum over $j$ then implies that $R_k$ will be large as soon as the difference between $\bar X_{1:j}$ and $\bar X_{j+1:k}$ is large for some $j$.

In the sequential context considered in this work, since $X_1, \dots, X_m$ is the learning sample known to be a stretch from a stationary time series, a first natural modification of~\eqref{eq:Rk} is to restrict the maximum over $j$ to $j \in \{m,\dots,k-1\}$. This is the idea considered by \cite{DetGos19} in a closed-end setting who, additionally, replaced the normalizing factor $k^{3/2}$ by $m^{3/2}$ so that the asymptotics of the corresponding monitoring scheme could be studied as the size $m$ of the learning sample tends to infinity. The resulting detector is
\begin{equation}
  \label{eq:Rm}
  R_m(k) =  \max_{m \leq j \leq k-1} \frac{j (k-j)}{m^{3/2}}  | \bar X_{1:j} - \bar X_{j+1:k} |, \qquad k \geq m+1.
\end{equation}
In an open-end setting, \cite{GosKleDet20} choose however not to consider the detector $R_m$ (see Remark~2.1 in the latter reference) but suggested  instead the detector
\begin{equation}
  \label{eq:Em}
  E_m(k) =  \max_{m \leq j \leq k-1} \frac{k-j}{m^{1/2}}  | \bar X_{1:j} - \bar X_{j+1:k} |, \qquad k \geq m+1.
\end{equation}
The difference between $R_m$ and $E_m$ evidently lies in the weighting of the absolute differences of means $|\bar X_{1:j} - \bar X_{j+1:k} |$, $j \in \{m,\dots,k-1\}$. Instead of weighting $|\bar X_{1:j} - \bar X_{j+1:k} |$ by $j (k-j) / m^2$,~\eqref{eq:Em} replaces $j/m$ by $1$. While this modification may be beneficial in terms of power when $k$ is close to $m$, it could have a negative impact when $k$ is substantially larger than $m$ because, then, $j (k-j) / m^2$ can be substantially larger than $(k-j) / m$. In other words, we suspect that, for changes not occurring at the beginning of the monitoring, a suitable detection scheme based on $R_m$ could be more powerful than the one proposed by \cite{GosKleDet20} based on $E_m$.

\begin{remark}
As mentioned in the introduction, the simplest detector in an open-end setting is probably the so-called \emph{ordinary CUSUM} initially considered by \cite{HorHusKokSte04} for investigating changes in the parameters of linear models. With the aim of detecting changes in the mean, it can be defined by
\begin{equation}
  \label{eq:Qm}
  Q_m(k) =  \frac{k-m}{m^{1/2}}  | \bar X_{1:m} - \bar X_{m+1:k} |, \qquad k \geq m+1.
\end{equation}
Following \cite{GosKleDet20}, we will use it as a benchmark in our Monte Carlo experiments.
\end{remark}

As explained in the introduction, the choice of a detector needs to be accompanied by the choice of a suitable threshold function. To heuristically justify our choice of a suitable threshold function for the detector $R_m$ in~\eqref{eq:Rm}, we momentarily consider the closed-end setting in which monitoring stops at the latest after observation $X_n$ is collected. Following \cite{DetGos19} among others, to be able to study the sequential testing scheme asymptotically, we set $n = \ip{mT}$ for some real $T > 1$. Then, under $H_0$ in~\eqref{eq:H0}, assuming that the  functional central limit theorem holds for $(X_i)_{i \in \Z}$, it can be verified using relatively simple arguments (see, e.g., \citealp{DetGos19}, Section~3, or \citealp{KojVer20a}, Section 2.3) that $\{R_m(\ip{mt})\}_{t \in [1,T]}$ converges in distribution to $\{L(t)\}_{t \in [1,T]}$, where
$$
L(t) = \sigma \sup_{1 \leq s \leq t} | tW(s) - sW(t)|, \qquad t \in [1,T],
$$
$W$ is a standard Brownian motion and $\sigma^2 = \sum_{i \in \Z} \Cov(X_0,X_i) > 0$ is the finite long-run variance of $(X_i)_{i \in \Z}$. For any fixed $t \in [1, T]$, by Brownian scaling and the substitution $u=s/t$,  $t^{-3/2} \sigma^{-1} L(t)$ is equal in distribution to
\begin{equation*}
\sup_{1 \leq s \leq t} \left| W\left( \frac{s}{t} \right) - \frac{s}{t} W(1) \right| = \sup_{1/t \leq u \leq 1} | W(u) - u W(1) |.
\end{equation*}
Hence, for large $t$, the distribution of $t^{-3/2}  \sigma^{-1} L(t)$ is close to that of the supremum of a Brownian bridge on $[0,1]$. As a consequence, under $H_0$ in~\eqref{eq:H0}, the distribution of $t^{-3/2}  \sigma^{-1} R_m(\ip{mt})$  stabilizes as $m$ and $t$ increase. The latter observation suggests the possibility of an open-end sequential testing scheme based on $R_m$ in~\eqref{eq:Rm} with a threshold function that is not too different from $t \mapsto t^{3/2}$.

As shall become clear from Theorem~\ref{thm:H0} below, in order to ensure that~\eqref{eq:typeIerror} is fully meaningful when $D_m = R_m$, the corresponding threshold function actually needs to diverge to $\infty$ (as $t\to \infty$) slightly faster than $t^{3/2}$. We propose to use as threshold function for~$R_m$
\begin{equation}
  \label{eq:wR}
  w_R(t)  = t^{3/2 + \eta} w_\gamma(t), \qquad t \in [1,\infty),
\end{equation}
where $\eta > 0$ is a real parameter and
\begin{equation}
  \label{eq:w:gamma}
  w_\gamma(t) = \max \left\{ \left(\frac{t-1}{t}\right)^\gamma, \epsilon \right\},  \qquad t \in [1,\infty),
\end{equation}
with $\gamma \geq 0$ another real parameter and $\epsilon > 0$ a technical constant that can be taken very small in practice (we used $\epsilon = 10^{-10}$ in our implementation).

\begin{figure}[t!]
\begin{center}
  \includegraphics*[width=1\linewidth]{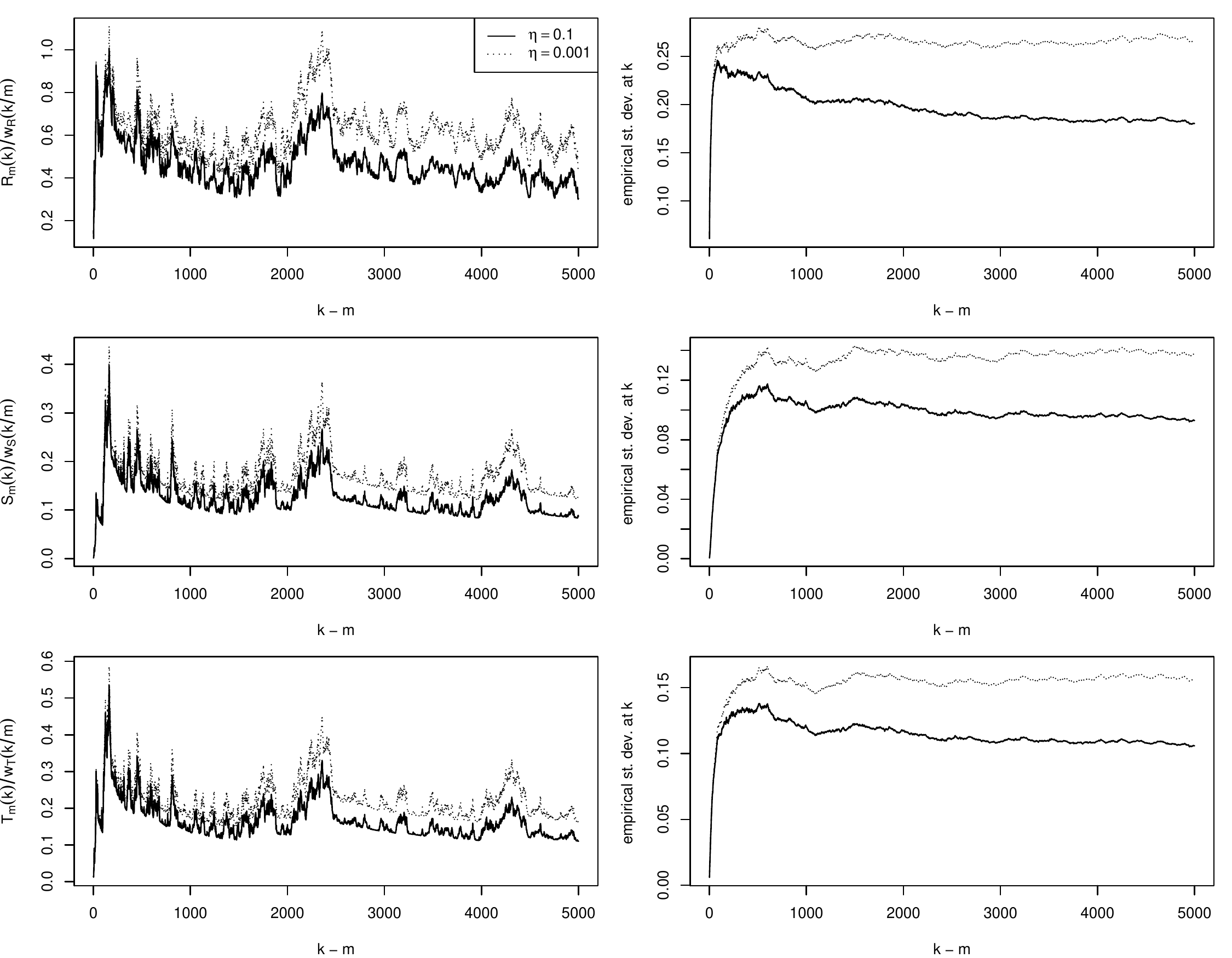}
  \caption{\label{fig:trajectories} Left column:  for $m=100$, $\gamma = 0$ and $\epsilon = 10^{-10}$, the solid lines represent one sample path of $\{\{ w_R(k/m) \}^{-1} R_m(k)\}_{m+1 \leq k \leq m+5000}$  (first row), $\{\{ w_S(k/m) \}^{-1} S_m(k)\}_{m+1 \leq k \leq m+5000}$ (second row) and $\{\{ w_T(k/m) \}^{-1} T_m(k)\}_{m+1 \leq k \leq m+5000}$ (third row) for $\eta=0.1$ computed from an independent sequence of standard normal random variables. The dotted lines represent the sample paths computed from the same sequence but with $\eta = 0.001$ instead. Right column: corresponding empirical standard deviations at $k$ against $k - m$ computed from 1000 sample paths.}
\end{center}
\end{figure}

Let us first explain the role of the parameter $\eta$. Following the perspective adopted in the discussion below~\eqref{eq:typeIerror}, the resulting monitoring can be seen as consisting of computing $\{ w_R(k/m) \}^{-1} R_m(k)$ for $k \geq m+1$. It is elementary that for any fixed $k\ge m+1$, as we increase $\eta$ both the mean and the variance of $\{ w_R(k/m) \}^{-1} R_m(k)$ decrease. The top-left plot of Figure~\ref{fig:trajectories} displays one sample path of $\{ \{ w_R(k/m) \}^{-1} R_m(k) \}_{m+1 \leq k \leq m+5000}$ for $m=100$, $\gamma=0$ and $\eta = 0.1$  computed from a sequence of independent standard normals (solid line) and the sample path  computed from the same sequence but with $\eta = 0.001$ instead (dotted line). Unsurprisingly, because of the factor $(k/m)^{-\eta}$ in $\{ w_R(k/m) \}^{-1} R_m(k)$, the sample path with $\eta = 0.1$ is below the sample path with $\eta = 0.001$. This effect of $\eta$ is confirmed by the top-right plot of Figure~\ref{fig:trajectories} which displays the corresponding empirical standard deviations at $k$ against $k-m$ computed from 1000 sample paths. As expected, increasing the parameter $\eta$ increases the rate of convergence (as $k \to \infty$) of $\{ w_R(k/m) \}^{-1} R_m(k)$ (and its mean and variance) to zero. Intuitively, in the context of open-end monitoring, one would therefore want $\eta$ to be very small so that there is little reduction in variability as time elapses. The practical choice of the parameter $\eta$ will be discussed in detail in Section~\ref{sec:quantiles}.

 Let us now explain the role of the parameter $\gamma$. The multiplication by the function $w_\gamma$ in~\eqref{eq:wR} aims at possibly improving the finite-sample performance of the sequential testing scheme at the beginning of the monitoring and has a negligible effect later. The use of such a modification is common in the literature and can be found for instance in \cite{HorHusKokSte04}, \cite{Fre15}, \cite{KirWeb18} and \cite{GosKleDet20}, among many others. Notice that, unlike what is frequently done in the literature, we do not impose that $\gamma$ be strictly smaller than $1/2$. To provide some further insight, consider the top-left plot of Figure~\ref{fig:R-gamma} which displays the empirical 95\% quantile (solid line), empirical standard deviation (dashed line) and sample mean (dotted line) of $\{ w_R(k/m) \}^{-1} R_m(k)$ for $m=100$, $\eta = 0.001$ and $\gamma = 0$ against $k - m$ computed from 1000 sample paths computed from independent standard normal sequences. As one can see, because of the small value of~$\eta$, the distribution of $\{ w_R(k/m) \}^{-1} R_m(k)$ appears to stabilize rather quickly as $k$ increases. The speed at which this occurs can be increased by increasing $\gamma$. For instance, by comparing the top-left and bottom-left plots of Figure~\ref{fig:R-gamma}, one can clearly see that the distribution of $\{ w_R(k/m) \}^{-1} R_m(k)$ stabilizes quicker as $k$ increases for $\gamma = 0.25$ than for $\gamma = 0$. However, as one can see from the plot for $\gamma = 0.45$, the value of $\gamma$ should not be taken too large.

\begin{figure}[t!]
\begin{center}
  \includegraphics*[width=1\linewidth]{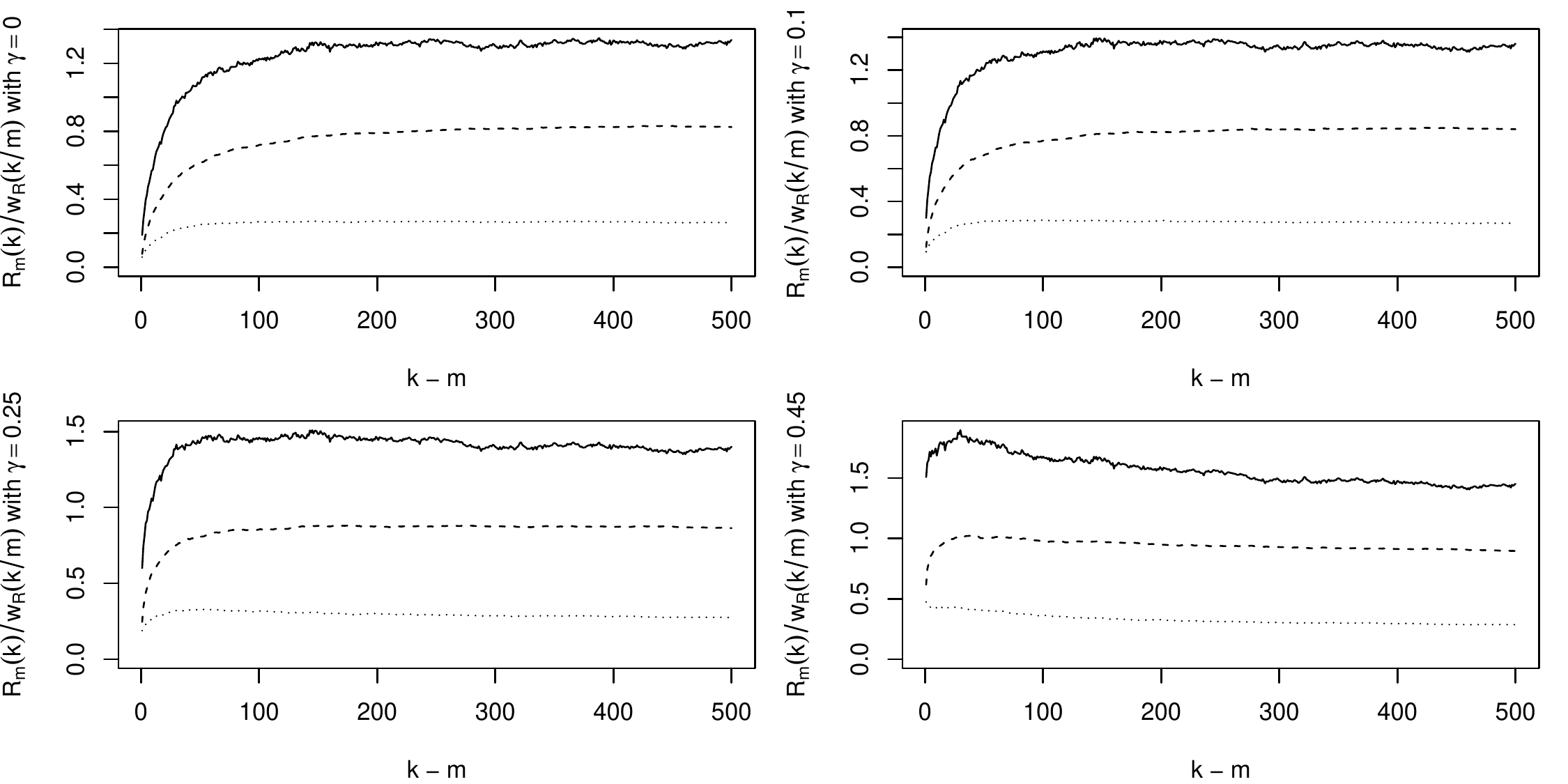}
  \caption{\label{fig:R-gamma} The solid (resp.\ dashed, dotted) line represents the empirical 95\% quantile (resp.\ empirical standard deviation, sample mean) of $\{ w_R(k/m) \}^{-1} R_m(k)$ with $m=100$, $\eta = 0.001$ and $\gamma \in \{0, 0.1, 0.25, 0.45\}$ against $k - m$ computed from 1000 sample paths.}
\end{center}
\end{figure}

In addition to the detector $R_m$ in~\eqref{eq:Rm}, we shall also consider the detectors
\begin{align}
  \label{eq:Sm}
  S_m(k) &=  \frac{1}{m} \sum_{j=m}^{k-1} \frac{j (k-j)}{m^{3/2}}  | \bar X_{1:j} - \bar X_{j+1:k} |, \qquad k \geq m+1, \\
  \label{eq:Tm}
  T_m(k) &=  \sqrt{ \frac{1}{m}  \sum_{j=m}^{k-1} \left\{ \frac{j (k-j)}{m^{3/2}}  ( \bar X_{1:j} - \bar X_{j+1:k} ) \right\}^2 }, \qquad k \geq m+1,
\end{align}
with corresponding threshold functions
\begin{align}
  \label{eq:wS}
  w_S(t) &= t^{5/2 + \eta} w_\gamma(t) , \\
  \label{eq:wT}
  w_T(t) &= t^{2 + \eta} w_\gamma(t).
\end{align}
As one can see, $S_m$ and $T_m$ could be regarded as the $L_1$ and $L_2$ versions, respectively, of $R_m$ in~\eqref{eq:Rm}.

The parameters $\eta$ and $\gamma$ in~\eqref{eq:wS} and~\eqref{eq:wT} play the same role as in~\eqref{eq:wR}. For~$\eta$, this can be empirically verified from the second and third rows of graphs in Figure~\ref{fig:trajectories}. For~$\gamma$, plots similar to Figure~\ref{fig:R-gamma} for $S_m$ and $T_m$ reveal that values of $\gamma$ larger than 0.25 seem meaningful for these two detectors. Specifically, it seems possible to improve the finite-sample performance of the corresponding schemes at the beginning of the monitoring by taking $\gamma$ as large as $0.85$ for $S_m$ and as large as $0.45$ for $T_m$.

\section{Asymptotics of the procedures}
\label{sec:asym}

To study the asymptotic behavior of the three considered monitoring schemes under $H_0$ in~\eqref{eq:H0}, we follow \cite{HorHusKokSte04}, \cite{AueHorHusKok06}, \cite{Fre15} and \cite{GosKleDet20}, among others, and assume that the observations satisfy the following condition.

\begin{cond}
  \label{cond:H0}
  The data are a stretch from a stationary time series $(X_i)_{i \in \Z}$ such that $\sigma^2 = \sum_{i \in \Z} \Cov(X_0,X_i)$, the long-run variance of $(X_i)_{i \in \Z}$, is strictly positive and finite. Furthermore, for all $m \in \N$, there exists two independent standard Brownian motions $W_{m,1}$ and $W_{m,2}$ such that, for some $0 < \xi < 1/2$,
\begin{equation}
  \label{eq:sup:cond}
  \sup_{m+1 \leq k \leq \infty} \frac{1}{(k-m)^\xi} \left| \sum_{i=m+1}^{k} \{X_i - \Ex(X_1) \}  - \sigma W_{m,1}(k-m) \right| = O_\Pr(1)
\end{equation}
and
\begin{equation}
  \label{eq:cond}
\frac{1}{m^\xi} \left| \sum_{i=1}^{m} \{X_i - \Ex(X_1) \}  - \sigma W_{m,2}(m) \right| = O_\Pr(1).
\end{equation}
\end{cond}

As mentioned in Remark~2.6 of \cite{GosKleDet20}, the validity of the previous conditions is discussed in Section~2 of \cite{AueHor04} for different classes of time series including GARCH and strongly mixing processes.

\begin{remark}
In the prototypical situation in which $(X_i)_{i \in \Z}$ is a sequence of independent normal random variables with variance $\sigma^2$, there exists a probability space on which the above conditions are trivially satisfied with $O_\Pr(1)$ replaced with 0.  This is essentially just the statement that the increments of a Brownian motion are standard normal random variables.  To be more precise, let $(B_j)_{j\ge 0}$ be independent standard Brownian bridges that are independent of $(X_i)_{i \in \Z}$. Without loss of generality, we may assume that the $X_i$ are standard normal, and we can define $W_{1,m}$ by first specifying its values at integer times by $W_{1,m}(k-m)=\sum_{i=m+1}^{k}X_i$, $k \geq m$, and then interpolating between these values  using the bridges $(B_{m+j})_{j\ge 0}$:  for $t \in (k-m,k-m+1)$, set $W_{1,m}(t)=W_{1,m}(k-m)+B_k\{t-(k-m)\}+\{t-(k-m)\} \{W_{1,m}(k-m+1)-W_{1,m}(k-m)\}$ (it is an exercise to check that the resulting process $W_{1,m}$ is a standard Brownian motion).  Then, the term in absolute values in~\eqref{eq:sup:cond} is exactly 0 for each $k$ and $m$. Similarly, setting $W_{2,m}(j)=\sum_{i=1}^j X_i$ for $j\le m$ and interpolating with the bridges $(B_j)_{0 \leq j \leq m-1}$ makes~\eqref{eq:cond} hold with 0 in the absolute value for each $m$. Furthermore, by construction, $W_{1,m}$ and $W_{2,m}$ are independent.
\end{remark}

The following result is proven in Appendix~\ref{app:H0}.

\begin{thm}
\label{thm:H0}
Under $H_0$ in~\eqref{eq:H0} and Condition~\ref{cond:H0}, for any fixed $\eta > 0$, $\epsilon > 0$ and $\gamma \geq 0$,
\begin{align*}
\sigma^{-1} \sup_{m+1 \leq k < \infty} \{ w_R(k/m) \}^{-1}  R_m(k) &\leadsto \sup_{1 \leq s \leq t < \infty} \{ w_R(t) \}^{-1} |t W(s) - s W(t) |, \\
\sigma^{-1} \sup_{m+1 \leq k < \infty} \{ w_S(k/m) \}^{-1}  S_m(k) &\leadsto \sup_{t \in [1,\infty)} \{ w_S(t) \}^{-1} \int_1^t |t W(s) - s W(t) | \dd s, \\
\sigma^{-1} \sup_{m+1 \leq k < \infty} \{ w_T(k/m) \}^{-1}  T_m(k) &\leadsto \sup_{t \in [1,\infty)} \{ w_T(t) \}^{-1} \sqrt{ \int_1^t \{ t W(s) - s W(t) \}^2 \dd s },
\end{align*}
where the arrow~`$\leadsto$' denotes convergence in distribution, the detectors $R_m$, $S_m$ and $T_m$ are defined in~\eqref{eq:Rm},~\eqref{eq:Sm} and~\eqref{eq:Tm}, respectively, the threshold functions $w_R$, $w_S$ and $w_T$ are defined in~\eqref{eq:wR},~\eqref{eq:wS} and~\eqref{eq:wT}, respectively, and $W$ is a standard Brownian motion. In addition, all the limiting random variables are almost surely finite.
\end{thm}

Imposing that $\eta$ is strictly positive in the previous theorem is necessary for ensuring that the limiting random variables are almost surely finite. Recall the definition of the function $w_\gamma$ in~\eqref{eq:w:gamma}. Since the function $t \mapsto 1/w_\gamma(t)$ is bounded from below by $1$, the following result, proven in Appendix~\ref{app:H0}, implies that for the monitoring scheme based on $R_m$ and $w_R$, this condition is necessary and sufficient.

\begin{prop}
  \label{prop:R0:infinite}
For any fixed $M > 0$,
\begin{equation*}
  \Pr \left\{ \sup_{1 \leq s \leq t < \infty} \frac{1}{t^{3/2}} | t W(s) - s W(t) | \geq M \right\} = 1,
\end{equation*}
where $W$ is a standard Brownian motion.
\end{prop}

\begin{remark}
We thus have that $\sup_{1 \leq s \leq t < \infty} \{ w_R(t) \}^{-1} |t W(s) - s W(t) |$ is almost surely finite for $\eta > 0$ and infinite for $\eta=0$. By the law of the iterated logarithm for Brownian motion, this supremum remains finite if $t^\eta$ in $w_R$ in~\eqref{eq:wR} is replaced by $h(t)$, where $h(t) = \sqrt{\log \log t}$ when $t>e^e$ and $h(t)=1$ when $t\le e^e$. We expect that Theorem~\ref{thm:H0} remains valid with such a modification which could be considered optimal in the sense that, as $t\to \infty$, $h$ diverges slower to infinity than $t \mapsto t^\eta$ for any $\eta > 0$. The latter implies that the use of $h(t)$ instead of $t^\eta$ entails the lowest possible variability reduction for the monitoring scheme (in the sense of the  discussion on the role of $\eta$ in the previous section) in the limit as $k \to \infty$. 
\end{remark}

The next corollary provides an operational version of Theorem~\ref{thm:H0} to carry out the three sequential change-point detection tests in practice.

\begin{cor}
  \label{cor:H0}
  The statement of Theorem~\ref{thm:H0} remains true if the long-run variance $\sigma^2$ is replaced by an estimator $\sigma_m^2$ of $\sigma^2$ computed from the learning sample $X_1,\dots,X_m$ such that $\sigma_m^2 - \sigma^2 = o_\Pr(1)$.
\end{cor}

To fix ideas, let us briefly explain how Corollary~\ref{cor:H0} can be used to carry out the sequential test based on $R_m$ in~\eqref{eq:Rm}. Given a significance level $\alpha \in (0,1/2)$, it is first necessary to accurately estimate $q_{R, 1-\alpha}$, the $(1-\alpha)$-quantile of the continuous random variable $\sup_{1 \leq s \leq t < \infty} \{ w_R(t) \}^{-1} |t W(s) - s W(t) |$ (this aspect will be discussed in detail in Section~\ref{sec:quantiles}). Then,  under $H_0$ in~\eqref{eq:H0}, from the Portmanteau theorem,  
\begin{multline}
  \label{eq:test}
  \Pr \left\{  \sigma_m^{-1} \sup_{m+1 \leq k < \infty} \{ w_R(k/m) \}^{-1}  R_m(k) \leq q_{R, 1-\alpha} \right\} \\ \to \Pr \left\{  \sup_{1 \leq s \leq t < \infty} \{ w_R(t) \}^{-1} |t W(s) - s W(t) | \leq q_{R, 1-\alpha} \right\} = 1 - \alpha.
\end{multline}
Hence, for large $m$, we can expect that, under  $H_0$, 
$$
\Pr\{   R_m(k) \leq \sigma_m q_{R, 1-\alpha} w_R(k/m) \text{ for all } k \geq m+1\}  \approx 1 - \alpha.
$$
In practice, after the arrival of observation $X_k$, $k > m$, $R_m(k)$ is computed from $X_1,\dots,X_k$ and compared to the threshold $\sigma_m q_{R, 1-\alpha} w_R(k/m)$ (or, equivalently, $\sigma_m^{-1} \{ w_R(k/m) \}^{-1}  R_m(k)$ is computed and compared to the threshold $q_{R, 1-\alpha}$). If $R_m(k) > \sigma_m q_{R, 1-\alpha} w_R(k/m)$, the null hypothesis is rejected and the monitoring stops. Otherwise, the next observation is collected and the previous iteration is carried out with the $k+1$ available data points. Corollary~\ref{cor:H0} and~\eqref{eq:test} in particular guarantee that such a sequential testing procedure will have asymptotic level~$\alpha$. Steps to carry out the tests based on $S_m$ in~\eqref{eq:Sm} and $T_m$ in~\eqref{eq:Tm} can be obtained \emph{mutatis mutandis}.

  We now turn to the asymptotic behavior of the monitoring schemes under sequences of alternatives related to $H_1$ in~\eqref{eq:H1}. We start with the procedure based on $R_m$ in~\eqref{eq:Rm} which we study under a condition similar to the one used by \citet[Theorem~2.13]{GosKleDet20}.

\begin{cond}
\label{cond:H1:R}
The data are a stretch, for some $m \in \N$, from the sequence of random variables $(X^{\sm}_i)_{i \in \N}$ defined by 
\begin{align*}
  X^{\sm}_i=\begin{cases}
    Y^{\sz}_i, \text{ if }i\le k^\star_m,\\
    Y^{\sm}_i, \text{ otherwise},
  \end{cases}
\end{align*}
where $(k^\star_m)_{m \in \N}$ is a sequence of integers such that $k^\star_m \geq m$ and $(Y^{\sz}_i)_{i \in \Z}$, $(Y^{\sss{1}}_i)_{i \in \Z}$, \dots, $(Y^{\sm}_i)_{i \in \Z},\dots$ are sequences of random variables defined on the same probability space and satisfying
\begin{enumerate}[(i)]
\item for each $m \geq 0$, $\big(\Ex(Y^{\sm}_i)\big)_{i \in \Z}$ is a constant sequence,
\item $\sqrt{m} |\Ex(Y^{\sz}_1)-\Ex(Y^{\sm}_1)| \to \infty$,
\item 
\begin{equation}
\label{eq:before}
\frac{1}{\sqrt{k^\star_m}} \sum_{i=1}^{k^\star_m} \{ Y_i^{\sz} - \Ex(Y_1^{\sz}) \} = O_\Pr(1),
\end{equation}
\end{enumerate}
and either 
\begin{enumerate}
\item[(iv)] there exist constants $c>0$ and $C_1>c_1>0$ such that $c_1\le m^{-1}k^\star_m\le C_1$ for all $m \in \N$ and  
\begin{equation}
  \label{eq:after1}
  \frac{1}{\sqrt{m}} \sum_{i=k^\star_m+1}^{k^\star_m + \ip{cm}} \{ Y_i^{\sm} - \Ex(Y^{\sm}_1) \} = O_\Pr(1),
\end{equation}
\end{enumerate}
or
\begin{enumerate}
\item[(v)] $k^\star_m/m \to \infty$, and there exists a constant $c>0$ such that 
\begin{equation}
  \label{eq:after2}
  \frac{1}{\sqrt{k^\star_m}} \sum_{i=k^\star_m+1}^{k^\star_m + \ip{ck^\star_m}} \{ Y_i^{\sm} - \Ex(Y^{\sm}_1) \} = O_\Pr(1).
\end{equation}
\end{enumerate}
\end{cond}

\begin{remark}
Statement $(iv)$ (resp.\ $(v)$) in Condition~\ref{cond:H1:R} is related to what were called \emph{early} (resp.\ \emph{late}) changes in~\cite{GosKleDet20} because $k^\star_m/m = O(1)$ (resp.\ $k^\star_m/m \to \infty$).
\end{remark}

\begin{remark}
Suppose that $(Y^{\sz}_i)_{i \in \Z}$ is a stationary centered sequence for which the central limit theorem holds, $(a_m)_{m \in \N}$ is a sequence such that $\sqrt{m} a_m \to \infty$ and, for any $m \in \N$, $Y^{\sm}_i = Y^{\sz}_i+a_m$, $i \in \Z$. It is then easy to verify that the sequences $(Y^{\sm}_i)_{i \in \Z}$, $m \geq 0$, satisfy $(i)$ and $(ii)$ in Condition~\ref{cond:H1:R}. Let us verify that they also satisfy~$(iii)$, $(iv)$ and~$(v)$. Let $\eps>0$.  Since the central limit theorem holds for $(Y^{\sz}_i)_{i \in \Z}$, there exists $M>0$ such that 
\begin{equation}
  \label{eq:unif:tightness}
\sup_{n\in \N} \Pr\Big( n^{-1/2} \sum_{i=1}^n Y^{\sz}_i>M \Big) < \eps.
\end{equation}
Hence,~\eqref{eq:before} holds for any sequence $(k^\star_m)_{m \in \N}$. For $c=1$, the quantity on the left-hand side of~\eqref{eq:after1} is equal to 
$$
\frac{1}{\sqrt{m}} \sum_{i=k^\star_m+1}^{k^\star_m + m}  Y_i^{\sz}.
$$
Since $\Pr(m^{-1/2}\sum_{i=1}^m Y^{\sz}_i > M) = \Pr(m^{-1/2}\sum_{i=n+1}^{n+m} Y_i^{\sz} > M)$ for all $m,n \in \N$,~\eqref{eq:unif:tightness} implies that $\sup_{n \in \N}\sup_{m \in \N} \Pr(m^{-1/2}\sum_{i=n+1}^{n+m} Y_i^{\sz}>M)<\eps$  which in turn implies that~\eqref{eq:after1} holds with $c=1$ for any sequence $(k^\star_m)_{m \in \N}$. Similarly, for $c=1$, the quantity on the left hand side of~\eqref{eq:after2} is equal to 
$$
\frac{1}{\sqrt{k^\star_m}} \sum_{i=k^\star_m+1}^{2k^\star_m}  Y_i^{\sz},
$$
and since $\sup_{n \in \N}\Pr(n^{-1/2}\sum_{i=n+1}^{2n} Y_i^{\sz}>M)<\eps$,~\eqref{eq:after2} holds with $c=1$ for any sequence $(k^\star_m)_{m \in \N}$.
\end{remark}

The following result proven in Appendix~\ref{app:H1} establishes the consistency of the procedure based on $R_m$ under Condition~\ref{cond:H1:R}.

\begin{thm}
  \label{thm:H1:R}
  Under Condition~\ref{cond:H1:R}, for any fixed $\eta \geq 0$, $\epsilon \geq 0$ and $\gamma \geq 0$,
  $$
  \sup_{m+1 \leq k < \infty} \{ w_R(k/m) \}^{-1}  R_m(k) \p \infty,
  $$
  where the arrow~`~$\p$' denotes convergence in probability.
\end{thm}

Assume that, under Condition~\ref{cond:H1:R}, the estimator $\sigma_m^2$ appearing in Corollary~\ref{cor:H0} converges in probability to a strictly positive constant. An immediate corollary of the previous theorem is then that, under Condition~\ref{cond:H1:R}, for any fixed $\eta \geq 0$, $\epsilon \geq 0$ and $\gamma \geq 0$, $\sigma_m^{-1} \sup_{m+1 \leq k < \infty} \{ w_R(k/m) \}^{-1}  R_m(k)$ diverges in probability to infinity. In other words, under the considered conditions, for any fixed $\eta \geq 0$, $\epsilon \geq 0$ and $\gamma \geq 0$, and any threshold $\kappa > 0$,
$$
\Pr\left[ \sigma_m^{-1} \sup_{m+1 \leq k < \infty} \{ w_R(k/m) \}^{-1}  R_m(k) > \kappa\right] \to 1.
$$

Proving the consistency of the procedures based on $S_m$ in~\eqref{eq:Sm} and $T_m$ in~\eqref{eq:Tm} for late changes turns out to be more difficult. The following result, proven in Appendix~\ref{app:H1}, shows that they are consistent for early changes. Note that the considered condition is very similar to those considered for instance in \citet[Theorem~3.8]{DetGos19} or \citet[Proposition~2.7]{KojVer20a} in the context of closed-end monitoring.

\begin{cond}
\label{cond:H1:ST}
The data are a stretch, for some $m \in \N$, from the sequence of random variables $(X^{\sm}_i)_{i \in \N}$ defined by 
\begin{align*}
  X^{\sm}_i=\begin{cases}
    Y_i, \text{ if }i\le k^\star_m,\\
    Z_i, \text{ otherwise},
  \end{cases}
\end{align*}
where $k^\star_m = \ip{mc}$ for some constant $c \geq  1$, and $(Y_i)_{i \in \Z}$ and $(Z_i)_{i \in \Z}$ are stationary sequences defined on the same probability space such that $\Ex(Y_1) \neq \Ex(Z_1)$ and for which the functional central limit theorem holds. 
\end{cond}

\begin{thm}
  \label{thm:H1:ST}
  Under Condition~\ref{cond:H1:ST}, for any fixed $T > c$, $\{ m^{-1/2} H_m(s,t) \}_{1 \leq s \leq t \leq T}$ converges in probability to $\{K_c(s,t)\}_{1 \leq s \leq t \leq T}$,  where, for any $1 \leq s \leq t$, 
\begin{align*}
  H_m(s,t) &= m^{-3/2} \ip{ms} (\ip{mt} - \ip{ms}) ( \bar X^{\sm}_{1:\ip{ms}} - \bar X^{\sm}_{\ip{ms}+1:\ip{mt}}), \\
  K_c(s,t) &= (s \wedge c) \{(t \vee c) - (s \vee c) \} \{\Ex(X^{\sm}_{k^\star}) - \Ex(X^{\sm}_{k^\star+1})\},
\end{align*}
and $\vee$ and $\wedge$ are the maximum and minimum operators, respectively. Consequently, for any fixed $\eta \geq 0$, $\epsilon \geq 0$ and $\gamma \geq 0$,
\begin{equation}
  \label{eq:H1:ST}
  \sup_{m+1 \leq k < \infty} \{ w_S(k/m) \}^{-1}  S_m(k) \p \infty  \quad \text{ and } \quad \sup_{m+1 \leq k < \infty} \{ w_T(k/m) \}^{-1}  T_m(k) \p \infty.
\end{equation}
\end{thm}

Since the estimator $\sigma_m^2$ appearing in Corollary~\ref{cor:H0} converges in probability to a strictly positive constant under Condition~\ref{cond:H1:ST}, an immediate corollary of Theorem~\ref{thm:H1:ST} is that, under the same conditions,
$$
\sigma_m^{-1} \sup_{m+1 \leq k < \infty} \{ w_S(k/m) \}^{-1}  S_m(k) \p \infty  \quad \text{ and } \quad \sigma_m^{-1} \sup_{m+1 \leq k < \infty} \{ w_T(k/m) \}^{-1}  T_m(k) \p \infty.
$$

\section{Estimation of high quantiles of the limiting distributions}
\label{sec:quantiles}

From the discussion following Corollary~\ref{cor:H0}, we know that accurate estimations of high quantiles of the limiting distributions appearing in Theorem~\ref{thm:H0} are necessary to carry out the sequential tests based on the detectors $R_m$, $S_m$ and $T_m$ defined in~\eqref{eq:Rm},~\eqref{eq:Sm} and~\eqref{eq:Tm}, respectively. Before we present the approach considered in this work, let us briefly explain how \cite{HorHusKokSte04} and \cite{GosKleDet20} proceeded for the procedure based on $Q_m$ in~\eqref{eq:Qm} and the one based on $E_m$ in~\eqref{eq:Em}, respectively. From the latter references, suitable threshold functions for these two detectors are $w_Q(t) = w_E(t) = t w_\gamma(t)$, $t \in [1,\infty)$, where the function $w_\gamma$ is defined in~\eqref{eq:wR} but with $\gamma$ restricted to the interval $[0,1/2)$. Furthermore, from \cite{GosKleDet20}, under $H_0$ in~\eqref{eq:H0} and Condition~\ref{cond:H0}, one has
\begin{align*}
  \sigma_m^{-1} \sup_{m+1 \leq k < \infty} \{ w_Q(k/m) \}^{-1}  Q_m(k) &\leadsto \sup_{t \in [0,1]} \frac{W(t)}{\max(t^\gamma, \epsilon)}, \\
  \sigma_m^{-1} \sup_{m+1 \leq k < \infty} \{ w_E(k/m) \}^{-1}  E_m(k) &\leadsto \sup_{0 \leq s \leq t \leq 1} \frac{1}{\max(t^\gamma, \epsilon)} |W(t) - W(s)|,
\end{align*}
where $W$ is a standard Brownian motion. Notice that, using the law of the iterated logarithm/local modulus of continuity for Brownian motion, it can be verified, since $\gamma < 1/2$, that the limiting random variables are almost surely finite even if $\epsilon$ is taken equal to zero. When $\gamma = 0$, it seems particularly natural to estimate high quantiles of these limiting distributions by Monte Carlo simulation using an approximation of $W$ on a fine grid of $[0,1]$ (one should keep in mind that although simulating such a process at discrete times always underestimates the true supremum, one can get arbitrarily close with probability as close to~1 as one wishes by taking a sufficiently fine grid). Note that in practice \cite{HorHusKokSte04} and \cite{GosKleDet20} use $\epsilon=0$ even when $\gamma > 0$. 

The following result, proven in Appendix~\ref{app:others}, suggests that the latter approach could be attempted for the first limiting distribution appearing in Theorem~\ref{thm:H0} related to the procedure based on $R_m$ in~\eqref{eq:Rm}.

\begin{prop}
  \label{prop:equiv:dist:R}
For any fixed $\eta > 0$, $\epsilon > 0$ and $\gamma \geq 0$, the random variable
\begin{equation}
  \label{eq:R:inf}
\sup_{1 \leq s \leq t < \infty} \{ w_R(t) \}^{-1} |t W(s) - s W(t) | = \sup_{1 \leq s \leq t < \infty} \frac{1}{t^{3/2+\eta} \max[ \{ (t-1)/t \}^\gamma, \epsilon]  } |t W(s) - s W(t) |
\end{equation}
and the random variable
\begin{equation}
  \label{eq:equiv:dist:R}
  \sup_{0<u \leq v\leq 1}\frac{u^{1/2+\eta}}{v\max\{(1-u)^\gamma,\epsilon\}}|W(v)-W(u)|
\end{equation}
are equal in distribution, where $W$ is a standard Brownian motion.
\end{prop}

Preliminary numerical experiments revealed however that there do not seem to be advantages to work with expressions of the form~\eqref{eq:equiv:dist:R} as these seem to transpose practical issues due to the presence of $\infty$ in~\eqref{eq:R:inf} into practical issues near zero. For that reason, we opted for a heuristic estimation approach which we detail in the rest of this section in the case of the procedure based on $R_m$ in~\eqref{eq:Rm} (and which we used \emph{mutatis mutandis} for the procedures based on $S_m$ in~\eqref{eq:Sm} and $T_m$ in~\eqref{eq:Tm} as well).

Fix $\eta > 0$, $\gamma \geq 0$ and $\alpha \in (0,1/2)$, and let $q_{R, 1-\alpha}$ denote the $(1-\alpha)$-quantile of $\sup_{1 \leq s \leq t < \infty} \{ w_R(t) \}^{-1} |t W(s) - s W(t) |$. To estimate $q_{R, 1-\alpha}$, we propose to take $m$ relatively large and simulate a large number $N$ of sample paths of $\{ \{ w_R(k/m) \}^{-1}  R_m(k) \}_{m+1 \leq k \leq m + 2^p}$ for $p$ large from sequences of independent standard normal random variables. At this point, it may be tempting to use the $(1-\alpha)$-empirical quantile of $\sup_{m+1 \leq k \leq m + 2^p} \{ w_R(k/m) \}^{-1}  R_m(k)$ as an estimate of  $q_{R, 1-\alpha}$. Although the latter is expected to be a good estimate of the $(1-\alpha)$-quantile of $\sup_{1 \leq s \leq t \leq T} \{ w_R(t) \}^{-1} |t W(s) - s W(t) |$ for $T = 1 + 2^p/m$, it may underestimate $q_{R,1-\alpha}$ since the supremum for $1 \le s \le t < \infty$ obviously dominates the supremum for $1 \le s \le t \le T$ for any $T>1$. The latter will actually depend on the value $\eta$. Indeed, recall from our discussion in Section~\ref{sec:det} that, because of the factor $(k/m)^{-\eta}$, increasing $\eta$ decreases the variance of $\{ w_R(k/m) \}^{-1} R_m(k)$ as $k$ increases. This has two consequences as far as the above mentioned empirical estimation of $q_{R,1-\alpha}$ is concerned: $(i)$ for any fixed $\eta > 0$, since the variability of $\{ w_R(k/m) \}^{-1} R_m(k)$ decreases as $k$ increases, if we could set $p$ large enough, we would be able to obtain a good estimate of $q_{R, 1-\alpha}$; unfortunately, our margin of action in that respect is limited by computational resources; $(ii)$ for any fixed $p$, as $\eta$ decreases to $0$, the probability that our estimate of the $(1-\alpha)$-quantile of $\sup_{1 \leq s \leq t \leq 1 + 2^p/m} \{ w_R(t) \}^{-1} |t W(s) - s W(t) |$ is also a good estimate of the $(1-\alpha)$-quantile of $\sup_{1 \leq s \leq t < \infty} \{ w_R(t) \}^{-1} |t W(s) - s W(t) |$ decreases to zero; in other words, for any fixed $p$, as $\eta$ decreases, it is less and less likely that the distribution of $\sup_{m+1 \leq k \leq m + 2^p} \{ w_R(k/m) \}^{-1}  R_m(k)$ is a good approximation of the distribution of $\sup_{1 \leq s \leq t < \infty} \{ w_R(t) \}^{-1} |t W(s) - s W(t) |$.
  
To try to empirically solve this quantile estimation problem, we propose, for any fixed $\eta > 0$, to attempt to model the relationship between $p$ and the $(1-\alpha)$-empirical quantile of $\sup_{m+1 \leq k \leq m + 2^p} \{ w_R(k/m) \}^{-1}  R_m(k)$, and to use the fitted model to extrapolate the value of the quantile for larger $p$. To do so, we set $m$ to 500, generated $N=15000$ sample paths of $\{ \{ w_R(k/m) \}^{-1}  R_m(k) \}_{m+1 \leq k \leq m + 2^{18}}$ using a computer cluster and, for $p \in \{10,\dots,18\}$, estimated the $(1-\alpha)$-quantiles of $\sup_{m+1 \leq k \leq m + 2^p} \{ w_R(k/m) \}^{-1}  R_m(k)$. Let $q_{R, 1-\alpha}^{(p)}$, $p \in \{10,\dots,18\}$, denote the resulting estimates. Then, we fitted a so-called \emph{asymptotic regression model} to the pairs $(p,q_{R, 1-\alpha}^{(p)})$, $p \in \{10,\dots,18\}$. The considered model, often used for the analysis of dose--response curves, is a three-parameter model with mean function $f(x) = c + (d-c)\{ 1-\exp(-x/e)\}$, where $y = d$ is the equation of the upper horizontal asymptote of $f$. Its fitting was carried out using the \textsf{R} package \texttt{drc} \citep{drc}. A candidate estimate of $q_{R, 1-\alpha}$, the $(1-\alpha)$-quantile of $\sup_{1 \leq s \leq t < \infty} \{ w_R(t) \}^{-1} |t W(s) - s W(t) |$, is then the resulting estimate of the parameter $d$.

\begin{figure}[t!]
  \begin{center}
    \includegraphics*[width=1\linewidth]{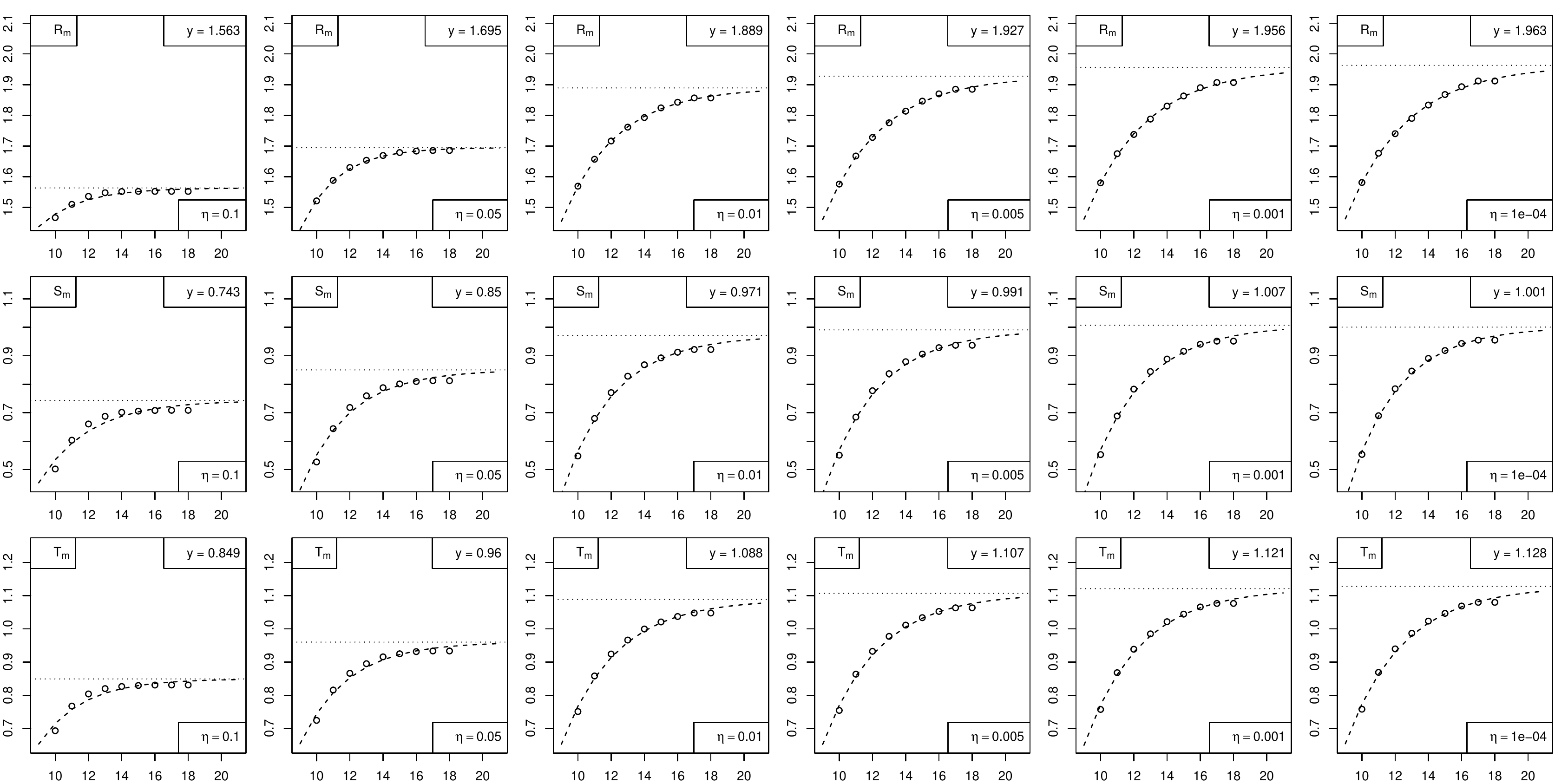}
    \caption{\label{fig:quantiles} For $\alpha = 0.05$ and $\gamma = 0$, scatter plots of $\{(p,q_{R, 1-\alpha}^{(p)})\}_{p \in \{10,\dots,18\}}$ (first row), $\{(p,q_{S, 1-\alpha}^{(p)})\}_{p \in \{10,\dots,18\}}$ (second row) and $\{(p,q_{T, 1-\alpha}^{(p)})\}_{p \in \{10,\dots,18\}}$ (third row), corresponding fitted asymptotic regression models (dashed curves) and estimates of the upper horizontal asymptotes (dotted lines) for $\eta \in \{0.1, 0.05, 0.01, 0.005, 0.001, 0.0001\}$.}
  \end{center}
\end{figure}

The first row of graphs in Figure~\ref{fig:quantiles} shows the scatter plots of $\{(p,q_{R, 1-\alpha}^{(p)})\}_{p \in \{10,\dots,18\}}$ for $\alpha = 0.05$, $\gamma = 0$ and $\eta \in \{0.1, 0.05, 0.01, 0.005, 0.001, 0.0001\}$. The corresponding fitted asymptotic regression models are represented by dashed curves. The estimated upper horizontal asymptotes are represented by dotted horizontal lines whose equations are given in the upper right corners of the plots.

As one can see from the first two graphs in the first row of Figure~\ref{fig:quantiles}, the scatter plots for $\eta = 0.1$ and $\eta = 0.05$ reveal the presence of a plateau for larger values of $p$. The latter is an empirical indication of the fact that the distribution of the random variable $\sup_{m+1 \leq k \leq m + 2^p} \{ w_R(k/m) \}^{-1}  R_m(k)$, say for $p = 18$, seems to be a good approximation of the distribution of $\sup_{m+1 \leq k < \infty} \{ w_R(k/m) \}^{-1}  R_m(k)$ and thus of $\sup_{1 \leq s \leq t < \infty} \{ w_R(t) \}^{-1} |t W(s) - s W(t) |$. In other words, because of the relatively large values of $\eta$ leading to a relatively quick reduction of the variability of $\{ w_R(k/m) \}^{-1}  R_m(k)$ as $k$ increases, the supremum of  $\{\{ w_R(k/m) \}^{-1}  R_m(k)\}_{m+1 \leq k < \infty}$ tends to occur for $k$ relatively close to $m$ (and smaller than $m + 2^{18}$). Specifically, under $H_0$ in~\eqref{eq:H0} and with $\eta = 0.1$ (resp.\ 0.05), the standard deviation of $\{ w_R(k/m) \}^{-1} R_m(k)$ at $k = 2^{18}$ could be very roughly approximated to be $(2^{18}/500)^{-0.1} \approx 53\%$ (resp.\ $(2^{18}/500)^{-0.05} \approx 73\%$) of the standard deviation of $\{ w_R(k/m) \}^{-1} R_m(k)$ in the early stages of the monitoring. The corresponding standard deviation reductions for $\eta = 0.01$, 0.005, 0.001 and 0.0001 are very roughly 7\%, 3\%, 1\% and less than 0.1\%, respectively.

The previous approximate variance reduction calculations explain in some sense why, in the first row of plots in Figure~\ref{fig:quantiles}, the smaller $\eta$, the larger the estimated quantiles $q_{R, 1-\alpha}^{(p)}$, $p \in \{10,\dots,18\}$, and thus the larger the estimated horizontal asymptote (which is a candidate estimate of $q_{R, 1-\alpha}$). As one can see, the rate of increase of the estimated quantiles decreases as $\eta$ decreases. For instance, there are hardly any differences between the plot for $\eta = 0.001$ and the plot for $\eta = 0.0001$. Unsurprisingly, this last plot is actually hardly indistinguishable from the plot for $\eta = 0$. The latter empirical observation practically implies that the proposed estimation technique will provide a finite estimate of $q_{R, 1-\alpha}$ even when, according to Proposition~\ref{prop:R0:infinite}, $q_{R, 1-\alpha}$ is known to be infinite. It follows that it is not meaningful in practice to consider values of $\eta$ that are ``very small''. Based on the previous approximate variance reduction calculations and the plots given in Figure~\ref{fig:quantiles}, our intuitive recommendation is not to take $\eta$ smaller than 0.001. Notice, that with $m=100$, this value of $\eta$ induces an approximate standard deviation reduction under $H_0$ for $\{ w_R(k/m) \}^{-1} R_m(k)$ after $10^{10}$ monitoring steps of less than 2\%. 

The estimated quantiles for $\eta = 0.001$, $\alpha \in \{0.01, 0.05, 0.1\}$, $\gamma = 0$ as well for the largest meaningful value of $\gamma$ (among those that were considered) for the three monitoring schemes are reported in Table~\ref{tab:quantiles:RST} along with standard errors. Estimated quantiles for larger values of $\eta$ are available in the \textsf{R} package \texttt{npcp} \citep{npcp}.

\input{quantiles_RST.tex}

\begin{remark}
It is a research project of its own to investigate more thoroughly the estimation of the quantiles both empirically and theoretically. For instance, one may wish to investigate bounds on the probability for a given $\eta > 0$ and $T > 1$ that $\sup_{1 \leq s \leq t \leq T} \{ w_R(t) \}^{-1} |t W(s) - s W(t) |$ occurs at some $(s,t)$ with $t > T$, and whether such occurrences tend to be  associated with small suprema or large suprema. 
\end{remark}

\section{Monte Carlo experiments}
\label{sec:MC}

To investigate the finite-sample properties of the studied open-end sequential change-point detection procedures, we carried out extensive Monte Carlo experiments. One should however keep in mind that numerical experiments cannot provide a full insight into the finite-sample behavior of open-end approaches as finite computing resources impose that monitoring has to be stopped eventually. 

To estimate the long-run variance $\sigma^2$ related to the learning sample, we used the approach of \cite{And91} based on the quadratic spectral kernel with automatic bandwidth selection as implemented in the function \texttt{lrvar()} of the \textsf{R} package \texttt{sandwich} \citep{Zei04}. We considered 10 data generating models, denoted M1, \dots, M10. Models M1, \dots, M5 are simple AR(1) models with normal innovations whose autoregressive parameter is equal to 0, 0.1, 0.3, 0.5 and 0.7, respectively. These models were chosen among others to allow a comparison of our results with those of \citet[Section~4.1]{GosKleDet20}. Model M6 generates independent observations from the Student $t$ distribution with 5 degrees of freedom. Model M7 is a GARCH(1,1) model with normal innovations and parameters $\omega = 0.012$, $\beta = 0.919$ and $\alpha = 0.072$ to mimic SP500 daily log-returns following \cite{JonPooRoc07}. Models M8 and M9 are the nonlinear autoregressive model used in \citet[Section 3.3]{PapPol01} and the exponential autoregressive model considered in \cite{AueTjo90} and \citet[Section 3.3]{PapPol01}. The underlying generating equations are
\begin{equation*}
X_i = 0.6 \sin( X_{i-1} ) + \epsilon_i
\end{equation*}
and
\begin{equation*}
X_i = \{ 0.8 - 1.1 \exp ( - 50 X_{i-1}^2 ) \} X_{i-1} + 0.1 \epsilon_i,
\end{equation*}
respectively, where the $\epsilon_i$ are independent standard normal innovations. Note that, for all time series models, a burn-out sample of 100 observations was used. Finally, in order to mimic count data, model M10 generates independent observations from a Poisson distribution with parameter 3.

In a first series of experiments, we attempted to assess how well the studied tests hold their level. To do so, we generated 5000 samples from models M1--M10 and used the first $m$ observations of each sample as learning sample. All the sequential tests were carried out at the 5\% nominal level using the estimated quantiles available in the \textsf{R} package \texttt{npcp} (see also Table~\ref{tab:quantiles:RST}). Because monitoring cannot go on indefinitely, we stopped the sequential testing procedures after the arrival of observation $n = m + 10000$. The percentages of rejection of $H_0$ in~\eqref{eq:H0} for the procedures based on $R_m$ in~\eqref{eq:Rm}, $S_m$ in~\eqref{eq:Sm}, $T_m$ in~\eqref{eq:Tm} with $\gamma = 0$ and $\eta \in \{0.05, 0.01, 0.005, 0.001\}$, as well as for the procedures based on $E_m$ in~\eqref{eq:Em} and $Q_m$ in~\eqref{eq:Qm} with $\gamma = 0$, are reported in Table~\ref{tab:H0} (to carry out the procedures based on $Q_m$ and $E_m$, we used the quantiles reported in Table 1 of \cite{HorHusKokSte04} and \cite{GosKleDet20}, respectively). Given the closed-end nature of the experiments, one should keep in mind that the empirical levels would have been higher had larger values of $n$ been considered. As one can see from Table~\ref{tab:H0}, for most models, the empirical levels drop below the 5\% threshold rather quickly as $m$ increases. When the tests are too liberal, it is probably mostly a consequence of the difficulty of the estimation of the long-run variance $\sigma^2$. It is for instance unsurprising that a large value of $m$ is necessary to obtain a reasonably accurate estimate of $\sigma^2$ for model M5 (the AR(1) model with the strongest serial dependence) or model M9 (an exponential autoregressive model). Notice that, overall, the tests based on $E_m$ and $Q_m$ are less liberal then the tests based on $R_m$, $S_m$ and $T_m$ when the serial dependence is strong and $m$ is small. The opposite tends to occur as $m$ increases. Among the three proposed procedures, the one based on $S_m$ is the most conservative, followed by the procedure based on $T_m$. The fact that the empirical levels are higher for large values of $\eta$ is due to the factor $(k/m)^{-\eta}$ which favors the occurrence of \emph{false alarms} (threshold exceedences) at the beginning of the monitoring. 

\input{H0.tex}

In a second series of experiments, we studied the finite-sample behavior of the tests under $H_1$ in~\eqref{eq:H1} using simulation scenarios similar to those considered in \citet[Section~4.1]{GosKleDet20}. We generated 1000 samples of size $n = m + 7000$ from models M1 with $m=100$ and M4 with $m=200$ and, for each sample, added a positive offset of $\delta$ to all observations after position $k^\star \in \{m, m + 500, m + 1000, m + 5000\}$. 
We started by investigating the influence of $\eta$ on the most conservative procedure according to Table~\ref{tab:H0}. The rejection percentages for the test based on $S_m$ with $\gamma = 0$ and $\eta \in \{0.1, 0.01, 0.001\}$, as well as for the procedures based on $E_m$ and $Q_m$ with $\gamma = 0$ are represented in Figure~\ref{fig:ts14S}. As one can see, when the change occurs right at the beginning of the monitoring, the larger $\eta$, the more powerful the procedure based on $S_m$. As $k^\star$ increases, the influence of the factor $(k/m)^{-\eta}$ comes into effect and the opposite tends to occurs (although the difference in power does not seem of practical importance for the monitoring periods under consideration). As far as the procedures based on $E_m$ and $Q_m$ are concerned, they are more powerful than the procedure based on $S_m$ when the change occurs at the beginning of the monitoring but, as expected, become less powerful as $k^\star$ increases. The fourth column of plots in Figure~\ref{fig:ts14S} shows that the difference in terms of power can be substantial. One should however keep in mind that, because of the factor $(k/m)^{-\eta}$, the test based on $E_m$ for instance may again become more powerful than the procedure based on $S_m$ for $k^\star$ extremely large. Nevertheless, even if we consider the least favorable setting ($\eta=0.1$), we believe that such a scenario is extremely unlikely to occur in practice given the (relatively) slow decrease of the function $t \mapsto t^{-\eta}$ and the fact that the weighting used in the definition of $E_m$ penalizes in some sense late changes as explained in the discussion below~\eqref{eq:Em}.

From the second row of plots of Figure~\ref{fig:ts14S}, we see that the previous conclusions seem to remain qualitatively the same when model M4 is used although, unsurprisingly, the stronger serial dependence gives the impression that the values of $k^\star$ are smaller.

\begin{figure}[t!]
\begin{center}
  \includegraphics*[width=1\linewidth]{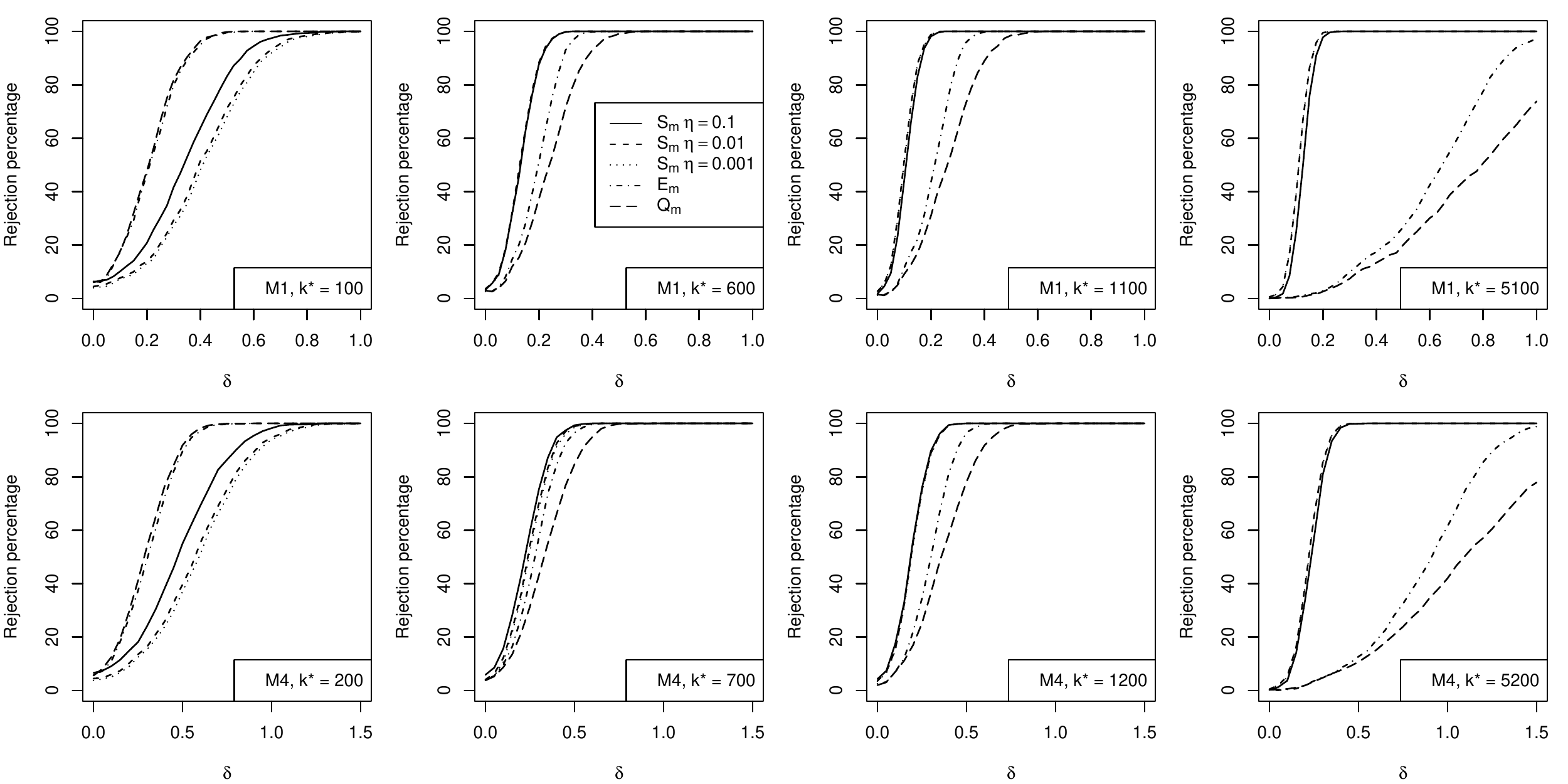}
  \caption{\label{fig:ts14S} Rejection percentages of $H_0$ in~\eqref{eq:H0} for the procedure based on $S_m$ with $\gamma = 0$ and $\eta \in \{0.1, 0.05, 0.01, 0.001\}$, as well as for the procedures based on $E_m$ and $Q_m$ with $\gamma = 0$ estimated from 1000 samples of size $n = m + 7000$ from model M1 with $m=100$ or M4 with $m=200$ such that, for each sample, a positive offset of $\delta$ was added to all observations after position $k^\star$.}
\end{center}
\end{figure}

Figure~\ref{fig:ts14all} reports the rejection percentages of $H_0$ in~\eqref{eq:H0} for the same experiment but for the procedures based on $R_m$, $S_m$, $T_m$ with $\gamma = 0$ and $\eta = 0.005$, as well as for the procedures based on $E_m$ and $Q_m$ with $\gamma = 0$. As one can see, among the three studied procedures, those based on $R_m$ and $T_m$ seem more powerful than the one based $S_m$ when the change occurs at the beginning of the monitoring (which is in accordance with the fact that the procedure based on $S_m$ is the most conservative as can be seen from Table~\ref{tab:H0}), while there seems to be very little difference between the three tests when the change occurs later. 

\begin{figure}[t!]
\begin{center}
  \includegraphics*[width=1\linewidth]{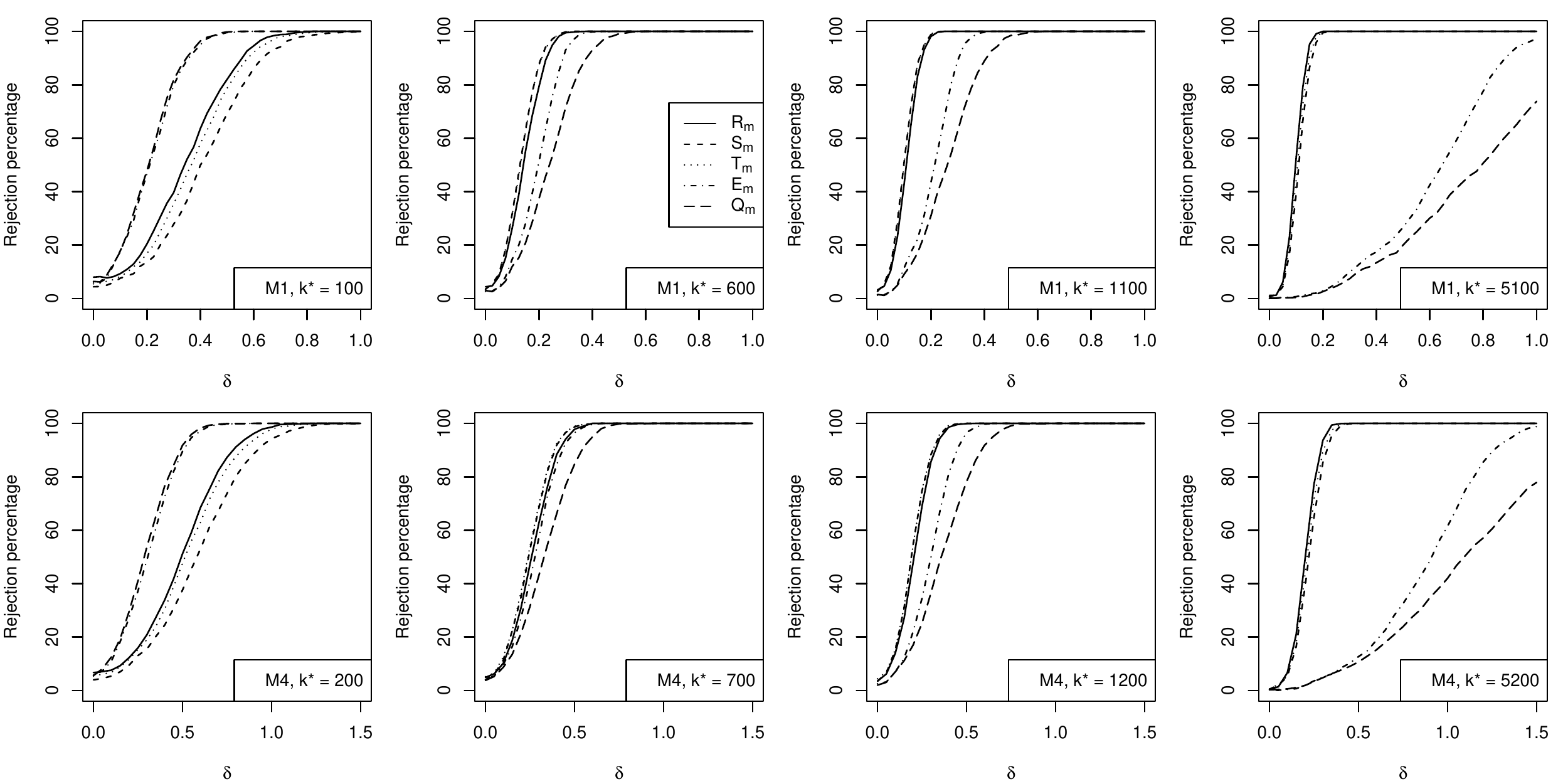}
  \caption{\label{fig:ts14all}  Rejection percentages of $H_0$ in~\eqref{eq:H0} for the procedures based on $R_m$, $S_m$, $T_m$ with $\gamma = 0$ and $\eta = 0.005$, as well as for the procedures based on $E_m$ and $Q_m$ with $\gamma = 0$ estimated from 1000 samples of size $n = m + 7000$ from model M1 with $m=100$ or M4 with $m=200$ such that, for each sample, a positive offset of $\delta$ was added to all observations after position $k^\star$.}
\end{center}
\end{figure}

The increase in power resulting from taking $\gamma$ equal to its largest meaningful value is illustrated in Figure~\ref{fig:ts14:gamma}. As expected, the improvement is visible only when changes occur at the beginning of the monitoring. In practice, we suggest to increase the value of $\gamma$ only when it is believed that the size $m$ of the learning sample permits a reasonably accurate estimation of the long-run variance $\sigma^2$.

\begin{figure}[t!]
\begin{center}
  \includegraphics*[width=1\linewidth]{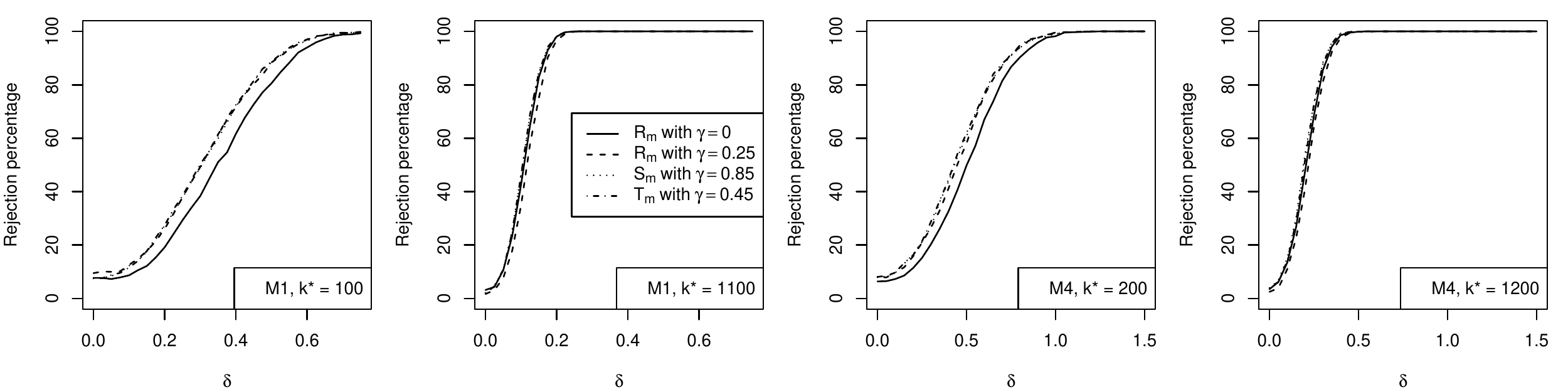}
  \caption{\label{fig:ts14:gamma} For $\eta = 0.001$, rejection percentages of $H_0$ in~\eqref{eq:H0} for the procedures based on $R_m$ with $\gamma = 0$, $R_m$ with $\gamma = 0.25$, $S_m$ with $\gamma = 0.85$ and $T_m$ with $\gamma = 0.45$ estimated from 1000 samples of size $n = m + 7000$ from model M1 with $m=100$ or M4 with $m=200$ such that, for each sample, a positive offset of $\delta$ was added to all observations after position $k^\star$.}
\end{center}
\end{figure}

Finally, we report the results of an experiment involving a longer monitoring period with late changes. Table~\ref{tab:H1} provides the percentages of rejection of $H_0$ in~\eqref{eq:H0} for the procedures based on $R_m$, $S_m$, $T_m$ with $\gamma = 0$ and $\eta \in \{0.1, 0.05, 0.01, 0.005, 0.001\}$ estimated from 2000 samples of size $n = 20000$ from model M1 with $m=100$ such that, for each sample, a positive offset of $\delta=0.1$ was added to all observations after position $k^\star=15000$. These results highlight again the role of $\eta$ and its influence on power through the factor $(k/m)^{-\eta}$. Notice that the corresponding rejection percentages of the procedures based $E_m$ and $Q_m$ are less than~1\%. 

\input{H1.tex}

Taking into account all the empirical results summarized in this section, we recommend to set $\eta$ to 0.005 or 0.001 and use either the procedure based on $T_m$ or the procedure based on $S_m$ as they are more conservative than the monitoring scheme based on $R_m$ while almost as powerful (except when changes occur at the beginning of the monitoring). If it is believed that the size $m$ of the learning sample permits a reasonably accurate estimation of the long-run variance $\sigma^2$, to improve the behavior of the procedures at the beginning of the monitoring, one can additionally set $\gamma$ to 0.45 for $T_m$ and 0.85 for $S_m$.

\section{Extensions to parameters whose estimators exhibit a mean-like behavior}
\label{sec:linear}

In their seminal work, \cite{GosKleDet20} actually considered monitoring schemes sensitive to changes in time series parameters whose estimators exhibit an asymptotic mean-like behavior. The aim of this section is to briefly demonstrate that the same type of extension is possible for the sequential procedures studied in this work. For the sake of keeping the notation simple, we restrict our discussion to univariate time series and univariate parameters.

Let $F$ be a univariate distribution function (d.f.) and let $\theta = \theta(F)$ be a univariate parameter of $F$ (such as the expectation, the variance, etc). Let $F_{j:k}$ be the empirical d.f.\ computed from the stretch $X_j,\dots,X_k$ of available observations. More formally, for any integers $j,k \geq 1$ and $x \in \R$, let
\begin{equation*}
  F_{j:k}(x) =
  \left\{
    \begin{array}{ll}
      \disp \frac{1}{k-j+1}\sum_{i=j}^k \1(X_i \leq x), \qquad & \text{if } j \leq k, \\
      0, \qquad &\text{otherwise},
    \end{array}
  \right.
\end{equation*}
and let $\theta_{j:k} = \theta(F_{j:k})$ be the corresponding plug-in estimator of $\theta$ computed from the stretch $X_j,\dots,X_k$. Natural extensions of the detectors $R_m$, $S_m$ and $T_m$ defined in~\eqref{eq:Rm},~\eqref{eq:Sm} and~\eqref{eq:Tm}, respectively, for monitoring changes in the parameter $\theta$ are then given, for any $k \geq m+1$, by
\begin{align}
  \label{eq:Rm:theta}
  R_m^\theta(k) &=  \max_{m \leq j \leq k-1} \frac{j (k-j)}{m^{3/2}}  | \theta_{1:j} - \theta_{j+1:k} |, \\
  \label{eq:Sm:theta}
  S_m^\theta(k) &=  \frac{1}{m} \sum_{j=m}^{k-1} \frac{j (k-j)}{m^{3/2}}  | \theta_{1:j} - \theta_{j+1:k} |, \\
  \label{eq:Tm:theta}
  T_m^\theta(k) &=  \sqrt{ \frac{1}{m}  \sum_{j=m}^{k-1} \left\{ \frac{j (k-j)}{m^{3/2}}  ( \theta_{1:j} - \theta_{j+1:k} ) \right\}^2 }.
\end{align}
Furthermore, assuming it exists, let
$$
\IF(x,F,\theta) = \lim_{\eps \downarrow 0} \frac{\theta\{(1-\eps) F + \eps \delta_x\} - \theta(F)}{\eps}
$$
denote the influence function related to $\theta$ and $F$ at $x \in \R$, where $\delta_x(\cdot) = \1(x \leq \cdot)$ is the d.f.\ of the Dirac measure at $x$. To be able to study the asymptotic validity of monitoring schemes based on $R_m^\theta$ in~\eqref{eq:Rm:theta}, $S_m^\theta$ in~\eqref{eq:Sm:theta} and $T_m^\theta$ in~\eqref{eq:Tm:theta}, we follow \cite{GosKleDet20} and focus on parameters $\theta$ that admit an asymptotic linearization in terms of the influence function, that is, such that
\begin{equation}
  \label{eq:linear}
  \theta_{j:k} - \theta = \theta(F_{j:k}) - \theta(F) = \frac{1}{k - j + 1} \sum_{i=j}^k \IF(X_i, F, \theta) + R_{j,k},
\end{equation}
where the remainders $R_{j,k}$ are asymptotically negligible in the sense of the following condition.

\begin{cond}
  \label{cond:remainders}
  The remainders in~\eqref{eq:linear} satisfy
  $$
  k^{-1/2} \max_{1 \leq i < j \leq k} (j-i+1) R_{i,j} \as 0 \qquad \text{as } k \to \infty,
  $$
  where the arrow~`~$\as$' denotes almost sure convergence.
\end{cond}

In the rest of this section, we assume that $H_0$ in~\eqref{eq:H0} holds. Moreover, for any integers $j,k \geq 1$, let
\begin{equation*}
  \bar \IF_{j:k} =
  \left\{
    \begin{array}{ll}
      \disp \frac{1}{k-j+1}\sum_{i=j}^k  \IF(X_i, F, \theta), \qquad & \text{if } j \leq k, \\
      0, \qquad &\text{otherwise}.
    \end{array}
  \right.
\end{equation*}
If the random variables $\IF(X_1, F, \theta),\dots,\IF(X_m, F, \theta),\IF(X_{m+1}, F, \theta), \dots$ were observable, one could naturally consider analogues of the detectors $R_m$, $S_m$ and $T_m$ in~\eqref{eq:Rm},~\eqref{eq:Sm} and~\eqref{eq:Tm}, respectively, defined, for $k \geq m+1$, by
\begin{align*}
  R_m^\IF(k) &=  \max_{m \leq j \leq k-1} \frac{j (k-j)}{m^{3/2}}  | \bar \IF_{1:j} - \bar \IF_{j+1:k} |, \\
  S_m^\IF(k) &=  \frac{1}{m} \sum_{j=m}^{k-1} \frac{j (k-j)}{m^{3/2}}  | \bar \IF_{1:j} - \bar \IF_{j+1:k} |, \\
  T_m^\IF(k) &=  \sqrt{ \frac{1}{m}  \sum_{j=m}^{k-1} \left\{ \frac{j (k-j)}{m^{3/2}}  ( \bar \IF_{1:j} - \bar \IF_{j+1:k} ) \right\}^2 }.
\end{align*}
Upon assuming that Condition~\ref{cond:H0} holds for the sequence $\big(\IF(X_i, F, \theta) \big)_{i \in \Z}$, one immediately obtains an analogue of Theorem~\ref{thm:H0} for the detectors $R_m^\IF$, $S_m^\IF$ and  $T_m^\IF$. The next result, proven in Appendix~\ref{app:others}, shows that, if Condition~\ref{cond:remainders} is additionally assumed, an analogue of Theorem~\ref{thm:H0} also holds for the computable detectors $R_m^\theta$ in~\eqref{eq:Rm:theta}, $S_m^\theta$ in~\eqref{eq:Sm:theta} and $T_m^\theta$ in~\eqref{eq:Tm:theta}.

\begin{prop}
  \label{prop:theta:IF}
Under $H_0$ in~\eqref{eq:H0} and Condition~\ref{cond:remainders}, for any $\eta > 0$, $\epsilon > 0$ and $\gamma \geq 0$,
\begin{align*}
  \sup_{m+1 \leq k \leq \infty}  \{ w_R(k/m) \}^{-1} | R_m^\theta(k) - R_m^\IF(k) | = o_\Pr(1), \\
  \sup_{m+1 \leq k \leq \infty}  \{ w_S(k/m) \}^{-1} | S_m^\theta(k) - S_m^\IF(k) | = o_\Pr(1), \\
  \sup_{m+1 \leq k \leq \infty}  \{ w_T(k/m) \}^{-1} | T_m^\theta(k) - T_m^\IF(k) | = o_\Pr(1),
\end{align*}
where the threshold functions $w_R$, $w_S$ and $w_T$ are defined in~\eqref{eq:wR},~\eqref{eq:wS} and~\eqref{eq:wT}, respectively.
\end{prop}

\begin{remark}
  As mentioned in \cite{GosKleDet20}, the verification of Condition~\ref{cond:remainders} is highly non-trivial. When $\theta$ is the variance or a quantile of $F$, it was shown to hold in probability in Section~4 of \cite{DetGos19}. In the multivariate parameter and time series case, it was verified in \citet[Section~3.2]{GosKleDet20} for a time-dependent linear model. In a related way, note that an inspection of the proof of Proposition~\ref{prop:theta:IF} reveals that Condition~\ref{cond:remainders} could actually be replaced by the requirement that the remainders in~\eqref{eq:linear} satisfy
  $$
  \sup_{m+1 \leq k < \infty} k^{-1/2} \max_{1 \leq i < j \leq k}(j-i+1) | R_{i,j} |  = o_\Pr(1).
  $$
\end{remark}

\section{Data example}
\label{sec:illus}

As a small data example, we consider a fictitious scenario consisting of monitoring global temperature anomalies for changes in the mean. Specifically, we use the time series of monthly global (land and ocean) temperature anomalies available at \url{http://www.climate.gov/} which covers the period January 1880 -- May 2020. The time series in degrees Celsius is represented in the left panel of Figure~\ref{fig:temperatures}. The solid vertical line marks the beginning of the fictitious monitoring and corresponds to September 1921 (and thus to a learning sample of size $m=500$). Note that this monitoring scenario is indeed fully fictitious, among other things, because temperature anomalies are computed with respect to the 20th century average \citep[see, e.g.,][]{SmiRey08} and, therefore, the corresponding time series would not have been available until the beginning of the current century.

The solid curve in the right panel of Figure~\ref{fig:temperatures} displays the evolution of the normalized detector $\sigma_m^{-1} \sup_{m+1 \leq k < \infty} \{ w_T(k/m) \}^{-1}  T_m(k)$ with $\eta=0.001$ and $\gamma = 0.45$ against $k \geq m+1$. The solid (resp.\ dashed) horizontal line represents the estimated 0.95-quantile (resp.\ 0.99-quantile) of the corresponding limiting distribution in Theorem~\ref{thm:H0}. The solid vertical line represents the date at which the normalized detector exceeded the aforementioned 0.95-quantile and corresponds to November 1939. Note that to estimate the point of change corresponding to an exceedence at position $k$, it seems natural to use
$$
\mathrm{argmax}_{m \leq j \leq k-1} \frac{j (k-j)}{m^{3/2}}  | \bar X_{1:j} - \bar X_{j+1:k} | + 1.
$$
The date of change corresponding to an exceedence in November 1939 is thus estimated to be April 1925 and is marked by a dashed vertical line in the right panel of Figure~\ref{fig:temperatures}.

\begin{figure}[t!]
\begin{center}
  \includegraphics*[width=1\linewidth]{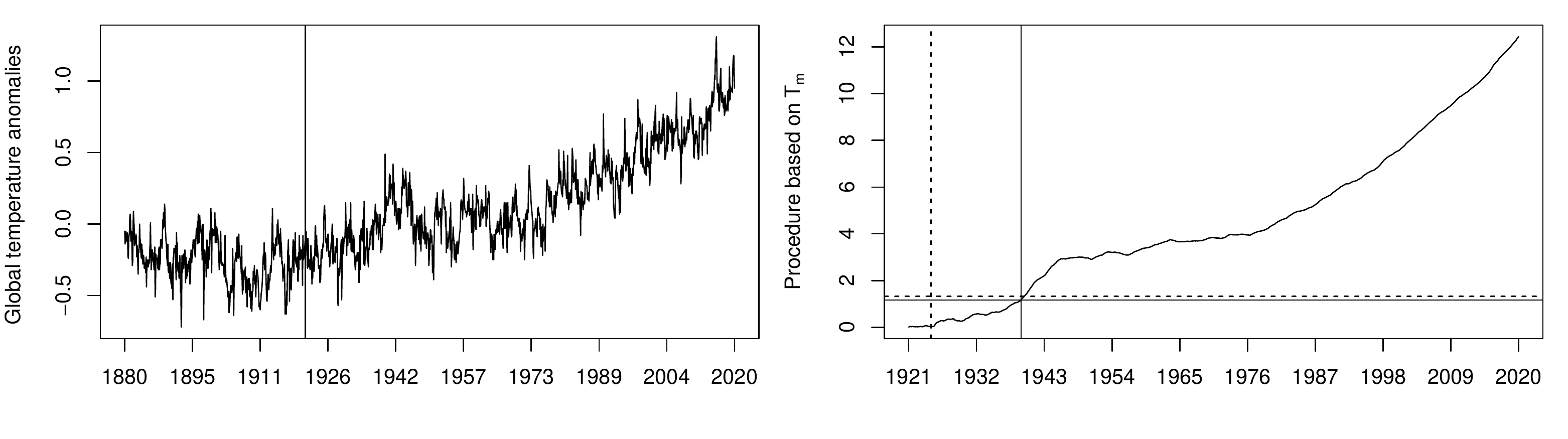}
  \caption{\label{fig:temperatures} Left: monthly global (land and ocean) temperature anomalies in degrees Celsius for the period January 1880 -- May 2020. The solid vertical line corresponds to September 1921 and marks the beginning of the fictitious monitoring. Right: the solid curve displays the normalized detector $\sigma_m^{-1} \sup_{m+1 \leq k < \infty} \{ w_T(k/m) \}^{-1}  T_m(k)$ with $\eta=0.001$ and $\gamma = 0.45$ against $k \geq m+1$. The solid (resp.\ dashed) horizontal line represents the estimated 0.95-quantile (resp.\ 0.99-quantile) of the corresponding limiting distribution in Theorem~\ref{thm:H0}. The solid (resp.\ dashed) vertical line corresponds to November 1939, the date of exceedence (resp.\ April 1925, the estimated date of change).}
\end{center}
\end{figure}

\section{Concluding remarks}

This work has demonstrated that it is relevant to define open-end sequential change-point tests such that the underlying detectors coincide with (or are related to) the retrospective CUSUM statistic at each monitoring step. From a practical perspective, when focusing on changes in the mean, such an approach was observed to lead to an increase in power with respect to existing competitors except when changes occur at the very beginning of the monitoring. Given the potentially very long term nature of open-end monitoring (by definition), it can be argued that having extra power for all but very early changes is strongly desirable. The price to pay for the additional power is a more complicated theoretical setting and the fact that quantiles of the underlying limiting distributions required to carry out the sequential tests in practice are harder to estimate. As far as monitoring for changes in other parameters than the mean is concerned, extensions are possible as long as the underlying estimators exhibit a mean-like asymptotic behavior as considered in \cite{GosKleDet20}.

\section*{Acknowledgments}

The authors would like to thank Ed Perkins for helpful discussions. MH is supported by Future Fellowship FT160100166 from the Australian Research Council.


\begin{appendix}

\section{Proofs of the results under $H_0$}
\label{app:H0}


The three following lemmas are used in the proof of Theorem~\ref{thm:H0}.

\begin{lem}
  \label{lem:ae:tilde}
Assume that Condition~\ref{cond:H0} holds and, for any $k \geq m+1$, let
\begin{align}
  \label{eq:tilde:Rm}
  \tilde R_m(k) &=  \sigma m^{-1/2} \max_{m \leq j \leq k-1} \left|  \frac{k}{m} \{ W_{m,2}(m) + W_{m,1}(j-m) \}  -  \frac{j}{m}  \{ W_{m,2}(m) + W_{m,1}(k-m) \} \right|, \\
  \label{eq:tilde:Sm}
  \tilde S_m(k) &=  \sigma m^{-1/2} \frac{1}{m} \sum_{j=m}^{k-1}  \left|  \frac{k}{m} \{ W_{m,2}(m) + W_{m,1}(j-m) \}  -  \frac{j}{m}  \{ W_{m,2}(m) + W_{m,1}(k-m) \} \right|, \\
  \label{eq:tilde:Tm}
  \tilde T_m(k) &=  \sigma m^{-1/2} \sqrt{ \frac{1}{m} \sum_{j=m}^{k-1}  \left[  \frac{k}{m} \{ W_{m,2}(m) + W_{m,1}(j-m) \}  -  \frac{j}{m}  \{ W_{m,2}(m) + W_{m,1}(k-m) \} \right]^2 }.
\end{align}
Then, for any fixed $\eta > 0$, $\epsilon > 0$ and $\gamma \geq 0$,
\begin{align}
  \label{eq:lem:tilde:R}
  \sup_{m+1 \leq k \leq \infty}  \{ w_R(k/m) \}^{-1} | R_m(k) - \tilde R_m(k) | &= o_\Pr(1), \\
  \label{eq:lem:tilde:S}
  \sup_{m+1 \leq k \leq \infty}  \{ w_S(k/m) \}^{-1} | S_m(k) - \tilde S_m(k) | &= o_\Pr(1), \\
  \label{eq:lem:tilde:T}
  \sup_{m+1 \leq k \leq \infty}  \{ w_T(k/m) \}^{-1} | T_m(k) - \tilde T_m(k) | &= o_\Pr(1),
\end{align}
where $w_R$, $w_S$ and $w_T$ are defined in~\eqref{eq:wR},~\eqref{eq:wS} and~\eqref{eq:wT}, respectively, and $R_m$, $S_m$ and $T_m$ are defined in~\eqref{eq:Rm},~\eqref{eq:Sm} and~\eqref{eq:Tm}, respectively.
\end{lem}

\begin{proof}
  Let us first show~\eqref{eq:lem:tilde:R}. From~\eqref{eq:Rm} and~\eqref{eq:tilde:Rm}, using the reverse triangle inequality for the maximum norm, 
  we have that, for any $k \geq m+1$, $\{ w_R(k/m) \}^{-1} | R_m(k) - \tilde R_m(k) | \leq U_m(k)$, where
\begin{multline}
  \label{eq:Um}
  U_m(k) = \epsilon^{-1} m^{-1/2} \left( \frac{k}{m} \right)^{-3/2-\eta}  \max_{m \leq j \leq k-1}  \left|  \frac{j (k-j)}{m}  \{ \bar X_{1:j} - \bar X_{j+1:k} \} \right. \\ \left. -  \frac{k}{m} \sigma \{ W_{m,2}(m) + W_{m,1}(j-m) \}  +  \frac{j}{m}  \sigma \{ W_{m,2}(m) + W_{m,1}(k-m) \} \right|,
\end{multline}
using the fact that $t \mapsto 1/w_\gamma(t)$ is bounded by $\epsilon^{-1}$. From~\eqref{eq:Xjk} and under Condition~\ref{cond:H0}, we obtain that, for any $k \geq m+1$ and $j \in \{m,\dots, k-1\}$,
\begin{align*}
  j (k-j) \{ \bar X_{1:j} - \bar X_{j+1:k} \} =& (k-j) \sum_{i=1}^j \{ X_i - \Ex(X_1) \} - j  \sum_{i=j+1}^k \{ X_i - \Ex(X_1) \} \\
                                              =& k \sum_{i=1}^j \{ X_i - \Ex(X_1) \} - j  \sum_{i=1}^k \{ X_i - \Ex(X_1) \} \\
  =& k \sum_{i=1}^m \{ X_i - \Ex(X_1) \} + k \sum_{i=m+1}^j \{ X_i - \Ex(X_1) \}  \\ &- j  \sum_{i=1}^m \{ X_i - \Ex(X_1) \} - j  \sum_{i=m+1}^k \{ X_i - \Ex(X_1) \}.
\end{align*}
Hence, by the triangle inequality, we have that
\begin{equation}
  \label{eq:sup:Um:conv}
\sup_{m+1 \leq k \leq \infty}  \{ w_R(k/m) \}^{-1} | R_m(k) - \tilde R_m(k) | \leq \sup_{m+1 \leq k \leq \infty} U_m(k) \leq \epsilon^{-1} ( I_m + I_m' + J_m + J_m' ),
\end{equation}
where
\begin{align}
  \nonumber 
  I_m = & m^{-1/2} \sup_{m+1 \leq k < \infty} \left(\frac{k}{m}\right)^{-1/2-\eta}  \max_{m \leq j \leq k-1}   \left| \sum_{i=1}^m \{ X_i - \Ex(X_1) \} - \sigma W_{m,2}(m) \right|, \\
  \nonumber 
  I_m' = & m^{-1/2} \sup_{m+1 \leq k < \infty} \left(\frac{k}{m}\right)^{-1/2-\eta}   \max_{m \leq j \leq k-1}  \left| \sum_{i=m+1}^j \{ X_i - \Ex(X_1) \} - \sigma W_{m,1}(j-m) \right|, \\
  \label{eq:three}
  J_m = & m^{-1/2} \sup_{m+1 \leq k < \infty} \left(\frac{k}{m}\right)^{-3/2-\eta}  \max_{m \leq j \leq k-1}  \frac{j}{m} \left| \sum_{i=1}^m \{ X_i - \Ex(X_1) \} - \sigma W_{m,2}(m) \right|, \\
  \label{eq:four}
  J_m' = &m^{-1/2} \sup_{m+1 \leq k < \infty} \left(\frac{k}{m}\right)^{-3/2-\eta} \max_{m \leq j \leq k-1} \frac{j}{m}  \left| \sum_{i=m+1}^k \{ X_i - \Ex(X_1) \} - \sigma W_{m,1}(k-m) \right|.
\end{align}
To prove~\eqref{eq:lem:tilde:R}, it remains to show that $I_m$, $I_m'$, $J_m$ and $J_m'$ converge to zero in probability. For any $\xi \in (0,1/2)$, 
$$
I_m = \frac{1}{m^\xi} \left| \sum_{i=1}^m \{ X_i - \Ex(X_1) \} - W_{m,2}(m) \right| m^{\xi + \eta} \sup_{m+1 \leq k < \infty} k^{-1/2 - \eta},
$$
which, under Condition~\ref{cond:H0}, is smaller than a term bounded in probability times $m^{\xi - 1/2}$, which converges to zero because $\xi < 1/2$. Similarly,
\begin{multline*}
I_m' \leq m^{\eta} \sup_{m+1 \leq k < \infty} k^{-1/2-\eta} \max_{m \leq j \leq k-1}   (j-m)^\xi \\ \times \sup_{m+1 \leq \ell < \infty} \frac{1}{(\ell-m)^\xi} \left| \sum_{i=m+1}^\ell \{ X_i - \Ex(X_1) \} - \sigma W_{m,1}(\ell-m) \right|.
\end{multline*}
Using again Condition~\ref{cond:H0}, the latter converges to zero in probability since
\begin{align*}
  m^{\eta} \sup_{m+1 \leq k < \infty} k^{- 1/2 - \eta}  \max_{m \leq j \leq k-1}  (j-m)^\xi &\leq m^{\eta} \sup_{m+1 \leq k < \infty} k^{-1/2 - \eta + \xi} \leq  m^{\xi -1/2} \to 0.
\end{align*}
The fact that $J_m$ in~\eqref{eq:three} and $J_m'$ in~\eqref{eq:four} converge to zero in probability can be checked by proceeding similarly. Hence, we have shown~\eqref{eq:lem:tilde:R}. It remains to prove~\eqref{eq:lem:tilde:S} and~\eqref{eq:lem:tilde:T}.

From~\eqref{eq:Sm} and~\eqref{eq:tilde:Sm}, using the reverse triangle inequality for the $L_1$ norm, it can be verified that, for any $k \geq m+1$,
\begin{multline*}
\{ w_S(k/m) \}^{-1} | S_m(k) - \tilde S_m(k) | \leq \epsilon^{-1} m^{-1/2} \left(\frac{k}{m}\right)^{-5/2-\eta} \frac{1}{m} \sum_{j=m}^{k-1} \left|  \frac{j (k-j)}{m}  \{ \bar X_{1:j} - \bar X_{j+1:k} \} \right. \\ \left. -  \frac{k}{m} \sigma \{ W_{m,2}(m) + W_{m,1}(j-m) \}  +  \frac{j}{m}  \sigma \{ W_{m,2}(m) + W_{m,1}(k-m) \} \right| \leq \frac{k-m}{k} U_m(k),
\end{multline*}
where $U_m(k)$ is defined in~\eqref{eq:Um}. Similarly, from~\eqref{eq:Tm} and~\eqref{eq:tilde:Tm}, using the reverse triangle inequality for the Euclidean norm,
\begin{multline*}
\{ w_T(k/m) \}^{-1} | T_m(k) - \tilde T_m(k) | \leq \epsilon^{-1} m^{-1/2} \left(\frac{k}{m}\right)^{-2-\eta}  \left( \frac{1}{m} \sum_{j=m}^{k-1} \left[  \frac{j (k-j)}{m}  \{ \bar X_{1:j} - \bar X_{j+1:k} \} \right. \right. \\ \left. \left. -  \frac{k}{m} \sigma \{ W_{m,2}(m) + W_{m,1}(j-m) \}  +  \frac{j}{m}  \sigma \{ W_{m,2}(m) + W_{m,1}(k-m) \} \right]^2 \right)^{1/2} \leq \sqrt{\frac{k-m}{k}} U_m(k).
\end{multline*}
The fact that~\eqref{eq:lem:tilde:S} and~\eqref{eq:lem:tilde:T} hold then immediately follows from the two previous displays,~\eqref{eq:sup:Um:conv} and the fact that $I_m$, $I_m'$, $J_m$ and $J_m'$ converge to zero in probability.
\end{proof}


\begin{lem}
\label{lem:unif:cont:R:eta}
For any fixed $\eta>0$, the random function $R_\eta$ defined by
\begin{equation}
  \label{eq:R:eta}
R_\eta(s,t) = \frac{1}{t^{3/2+\eta}} | (t-s) W_2(1) + t W_1(s-1) - s W_1(t-1) |, \qquad 1 \leq s \leq t<\infty,
\end{equation}
where $W_1$ and $W_2$ are independent standard Brownian motions, is almost surely bounded and uniformly continuous.
\end{lem}
\begin{proof}
Fix $\eta > 0$ and let us first verify that $R_\eta$ is almost surely bounded on $A = \{(s,t)\in [1,\infty)^2: s \le t\}$. By the triangle inequality, $\sup_{1 \leq s \leq t <\infty} R_\eta(s,t)$ is smaller than
\begin{align*}
  \sup_{1 \leq s \leq t <\infty} &t^{-3/2 - \eta}  (t-s) | W_2(1) | + \sup_{1 \leq s \leq t <\infty} t^{-1/2 - \eta} | W_1(s-1) |  +  \sup_{1 \leq s \leq t <\infty} t^{-3/2 - \eta} s | W_1(t-1) | \\
                                 &\leq | W_2(1) | \sup_{1 \leq s \leq t <\infty} t^{-1/2 - \eta} + \sup_{1 \leq s < \infty} s^{-1/2 - \eta} | W_1(s-1) |  +   \sup_{1 \leq t <\infty} t^{-1/2 - \eta} | W_1(t-1) | \\
                                 &\leq | W_2(1) |  + 2 \sup_{1 \leq t <\infty} t^{-1/2 - \eta} | W_1(t-1) |.
\end{align*}
It thus remains to verify that $\sup_{1 \leq t <\infty} t^{-1/2 - \eta} | W_1(t-1) |$ is almost surely bounded. From the law of the iterated logarithm for Brownian motion, we have that, almost surely,
\begin{align*}
  \limsup_{t \to \infty} & \, t^{-1/2 - \eta} | W_1(t-1) | =  \lim_{T \to \infty} \sup_{T \leq t < \infty} t^{-1/2 - \eta} | W_1(t-1) | =  \lim_{T \to \infty} \sup_{T \leq t < \infty} (t+1)^{-1/2 - \eta} | W_1(t) | \\ &\leq \lim_{T \to \infty} \sup_{T \leq t < \infty} (2 t \log \log t)^{-1/2} | W_1(t) |  \times \lim_{T \to \infty} \sup_{T \leq t < \infty}(2 t \log \log t)^{1/2} (t+1)^{-1/2 - \eta} \\  &= 1 \times \lim_{t \to \infty} (2 t \log \log t)^{1/2} (t+1)^{-1/2-\eta} = 0.
\end{align*}
Hence, on one hand, there exists $T \in (1,\infty)$ such that $\sup_{T \leq t < \infty} t^{-1/2 - \eta} | W_1(t-1) | \leq 1$ almost surely. On the other hand, since $t \mapsto t^{-1/2 - \eta} | W_1(t-1) |$ is a process whose sample paths are almost surely continuous,  $\sup_{1 \leq t  \leq T} < t^{-1/2 - \eta} | W_1(t-1) | < \infty$ with probability one.

It remains to prove that $R_\eta$ is almost surely uniformly continuous on $A = \{(s,t)\in [1,\infty)^2: s \le t\}$. Let $\eps > 0$. We must show that there exists some (random) $\delta>0$ such that, for all $(s,t),(s',t') \in A$ such that $d((s,t),(s',t')) < \delta$, $|R_\eta(s,t)-R_\eta(s',t')| \le \eps$.

By the law of the iterated logarithm, there exists $T_\eps\ge 2$ random such that
\begin{equation}
  \label{eq:hi1}
  \sup_{u\ge T_\eps}\frac{|W_1(u)|}{u^{1/2+\eta}}\le \frac{\eps}{8},
\end{equation}
and since $W_2(1)$ is almost surely finite we can choose $T_\eps$ to also satisfy 
\begin{equation}
  \label{eq:hi2}
  \frac{|W_2(1)|}{(T_\eps)^{1/2+\eta}}\le \frac{\eps}{8}.
\end{equation}
Since $[1,T_\eps+3]$ is compact, there exists $T'_\eps>T_\eps+5$ (also random) such that
\begin{align}
  \label{eq:hi3}
\sup_{u\in [1, T_\eps+3]}\frac{|W_1(u)|}{(T'_\eps)^{1/2+\eta}}<\frac{\eps}{8}. 
\end{align}
We will consider the following three subsets of $A$ whose union is $A$:
\begin{align*}
A_0&=\{(s,t)\in A: s\ge T_\eps+2\},\\
A_1&=\{(s,t)\in A: t\le T'_\eps+1\},\\
A_2&=\{(s,t)\in A: s\le T_\eps+2, t\ge T'_\eps+1\},
\end{align*}
and find a single $\delta \le 1$ that works no matter which of the above $(s,t)$ is in.

Let $(s,t)\in A_0$, and let $(s',t')$ be such that $d((s,t),(s',t'))\le 1$ (so $s'\ge T_\eps+1$).  Then
\begin{align}
|R_\eta(s,t) - R_\eta(s',t')|&\le \frac{1}{t^{3/2+\eta}}\left|(t-s)W_2(1)+tW_1(s-1)-sW_1(t-1)\right|\label{eq:bye1}\\
                             &\quad +\frac{1}{t'^{3/2+\eta}}\left|(t'-s')W_2(1)+t'W_1(s'-1)-s'W_1(t'-1)\right|\label{eq:bye2}\\
  \nonumber
&\le \frac{2|W_2(1)|}{(T_\eps)^{1/2+\eta}}+\frac{|W_1(s-1)|}{(s-1)^{1/2+\eta}}+\frac{|W_1(s'-1)|}{(s'-1)^{1/2+\eta}}+\frac{|W_1(t-1)|}{(t-1)^{1/2+\eta}}+\frac{|W_1(t'-1)|}{(t'-1)^{1/2+\eta}}.
\end{align}
By~\eqref{eq:hi1} and~\eqref{eq:hi2}, this is at most $6\eps/8$.

Let $(s,t)\in A_2$ and $(s',t')$ be such that $d((s,t),(s',t'))\le 1$ (so $s'\le T_\eps+3$).  Then, using~\eqref{eq:bye1} and~\eqref{eq:bye2} again, we get that
\begin{align*}
|R_\eta(s,t) - R_\eta(s',t')|&\le \frac{2|W_2(1)|}{(T'_\eps)^{1/2+\eta}}+\frac{|W_1(s-1)|}{(T'_\eps)^{1/2+\eta}}+\frac{|W_1(s'-1)|}{(T'_\eps)^{1/2+\eta}}+\frac{|W_1(t-1)|}{(t-1)^{1/2+\eta}}+\frac{|W_1(t'-1)|}{(t'-1)^{1/2+\eta}}.
\end{align*}
Using~\eqref{eq:hi1},~\eqref{eq:hi3} and~\eqref{eq:hi2}, we get that this is at most $6\eps/8$.

Finally, suppose that $(s,t)\in A_1$.  Since the set $A_1' = \{(s,t)\in A:t \le T'_\eps+2\}$ (note that $A_1'\supset A_1$) is compact and $R_\eta$ is continuous, it follows that $R_\eta$ is uniformly continuous on $A'_1$, so there exists some (random) $\delta_0>0$ such that whenever $(s,t),(s',t')\in A'_1$ and $d((s,t),(s',t'))\le \delta_0$, then we have  $|R_\eta(s,t)-R_\eta(s',t')|\le \eps/2$.  Note that if $(s,t) \in A_1$ and $d((s,t),(s',t'))\le 1$, then $(s,t)$ and $(s',t')$ are both in $A'_1$.

Let $\delta=\min(\delta_0,1)$.  Then, whenever $d((s,t),(s',t'))\le \delta$, we have that $(s,t)$ is in (at least) one of $A_0,A_1,A_2$, so by the above $|R_\eta(s,t) - R_\eta(s',t')|< \eps$. 
\end{proof}


\begin{lem}
  \label{lem:wc:tilde}
For any fixed $\eta > 0$, $\epsilon > 0$ and $\gamma \geq 0$,
\begin{align}
  \label{eq:wc:R}
  \sup_{m+1 \leq k < \infty}  \{ w_R(k/m) \}^{-1} \tilde R_m(k) &\leadsto \sigma \sup_{1 \leq s \leq t <\infty} \{ w_\gamma(t) \}^{-1} R_\eta(s,t), \\
  \label{eq:wc:S}
  \sup_{m+1 \leq k < \infty} \{ w_S(k/m) \}^{-1}  \tilde S_m(k) &\leadsto  \sigma \sup_{1 \leq t <\infty}  \{ w_\gamma(t) \}^{-1} t^{-1}  \int_1^t |  R_\eta(s,t) | \dd s, \\
  \label{eq:wc:T}
\sup_{m+1 \leq k < \infty} \{ w_T(k/m) \}^{-1}  \tilde T_m(k) &\leadsto  \sigma \sup_{1 \leq t <\infty} \{ w_\gamma(t) \}^{-1} t^{-1/2} \sqrt{ \int_1^t \{ R_\eta(s,t) \}^2 \dd s},
\end{align}
where $w_R$, $w_S$ and $w_T$ are defined in~\eqref{eq:wR},~\eqref{eq:wS} and~\eqref{eq:wT}, respectively, $\tilde R_m$, $\tilde S_m$ and $\tilde T_m$ are defined in~\eqref{eq:tilde:Rm},~\eqref{eq:tilde:Sm} and~\eqref{eq:tilde:Tm}, respectively, $w_\gamma$ is defined in~\eqref{eq:w:gamma} and the random function $R_\eta$ is defined in~\eqref{eq:R:eta}. In addition, all the limiting random variables are almost surely finite.
\end{lem}

\begin{proof}
Fix $\eta > 0$, $\epsilon > 0$ and $\gamma \geq 0$ and let $W_1$ and $W_2$ be independent standard Brownian motions. Then, from~\eqref{eq:tilde:Rm}, for any $m \in \N$, $\sup_{m+1 \leq k < \infty}  \{ w_R(k/m) \}^{-1} \tilde R_m(k)$ is equal in distribution to
\begin{multline*}
\sigma m^{-1/2} \sup_{m+1 \leq k < \infty}  \{ w_R(k/m) \}^{-1}  \max_{m \leq j \leq k-1} \left|  \frac{k}{m} \{ W_2(m) + W_1(j-m) \} \right. \\ \left. -  \frac{j}{m}  \{ W_2(m) + W_1(k-m) \} \right|.
\end{multline*}
Next, notice that, for any $k \geq m+1$ and any $j \in \{m,\dots,k-1\}$, there exists $1 \leq s \leq t$ such that $k = \ip{mt}$ and $j = \ip{ms}$. Hence, the previous expression can be rewritten as
\begin{multline*}
  \sigma m^{-1/2} \sup_{t \in [1, \infty)}  \{ w_R(\ip{mt}/m) \}^{-1}  \sup_{s \in [1,t]} \left|  \frac{\ip{mt}}{m} \{ W_2(m) + W_1(\ip{ms}-m) \}  \right. \\ \left. -  \frac{\ip{ms}}{m}  \{ W_2(m) + W_1(\ip{mt}-m) \} \right|.
\end{multline*}
By Brownian scaling, the latter is equal in distribution to
\begin{multline*}
\sigma \sup_{1 \leq s \leq t <\infty}  \{ w_R(\ip{mt}/m) \}^{-1}  \left|  \frac{\ip{mt}}{m} \left\{ W_2(1) + W_1\left(\frac{\ip{ms}}{m}-1\right) \right\}  \right. \\ \left. -  \frac{\ip{ms}}{m}  \left\{ W_2(1) + W_1\left(\frac{\ip{mt}}{m}-1\right) \right\} \right|,
\end{multline*}
which, using~\eqref{eq:wR},~\eqref{eq:w:gamma} and the function $R_\eta$ defined in~\eqref{eq:R:eta}, can be expressed as
$$
\sigma  \sup_{1 \leq s \leq t <\infty}   \{ w_\gamma(\ip{mt}/m) \}^{-1} R_\eta(\ip{ms}/m, \ip{mt}/m).
$$
Using additionally the fact that, for any functions $f, g$,
\begin{equation}
  \label{eq:ineq:fg}
  \big| \sup_x | f(x) | - \sup_x | g(x)| \big| \leq  \sup_x | f(x) - g(x) |,  
\end{equation}
we obtain that
\begin{align}
  \nonumber
  \Big|  \sup_{1 \leq s \leq t <\infty}  &\{ w_\gamma(\ip{mt}/m) \}^{-1} R_\eta(\ip{ms}/m, \ip{mt}/m) -  \sup_{1 \leq s \leq t <\infty} \{ w_\gamma(t) \}^{-1} R_\eta(s,t) \Big| \\
  \nonumber
                                         &\leq \sup_{1 \leq s \leq t <\infty}  \Big| \{ w_\gamma(\ip{mt}/m) \}^{-1} R_\eta(\ip{ms}/m, \ip{mt}/m) - \{ w_\gamma(t) \}^{-1} R_\eta(s,t) \Big| \\
  \label{eq:modulus:cont}
                                         &\leq   \sup_{1 \leq s \leq t <\infty, 1 \leq s' \leq t' <\infty \atop |s - s'| \leq 1/m, |t - t'| \leq 1/m}  \Big| \{ w_\gamma(t') \}^{-1} R_\eta(s', t') - \{ w_\gamma(t) \}^{-1} R_\eta(s,t) \Big|,
\end{align}
since $\sup_{t \in [1, \infty)} | \ip{mt}/m - t | \leq 1/m$. 

From Lemma~\ref{lem:unif:cont:R:eta}, we know that $R_\eta$ is almost surely bounded and uniformly continuous on $\{(s,t)\in [1,\infty)^2: s\le t\}$. Furthermore, the function $t \mapsto 1/w_\gamma(t)$ being bounded, continuous and converging to zero as $t \to \infty$, it is also uniformly continuous on $[1,\infty)$. The latter facts imply that the function $(s,t) \mapsto \{ w_\gamma(t) \}^{-1} R_\eta(s,t)$ is almost surely bounded and uniformly continuous on $\{(s,t)\in [1,\infty)^2: s\le t\}$ and, therefore, that~\eqref{eq:modulus:cont} converges almost surely to zero and, finally, that~\eqref{eq:wc:R} holds with the limit being almost surely finite.

Let us now show~\eqref{eq:wc:S}. First, notice that the limiting random variable is almost surely finite as an immediate consequence of the inequality
$$
\sup_{1 \leq t <\infty} t^{-1}  \int_1^t |  R_\eta(s,t) | \dd s \leq \sup_{1 \leq s \leq t <\infty} R_\eta(s,t).
$$
Then, from~\eqref{eq:tilde:Sm}, for any $m \in \N$, $\sup_{m+1 \leq k < \infty}  \{ w_S(k/m) \}^{-1} \tilde S_m(k)$ is equal in distribution to
$$
\sigma m^{-1/2} \sup_{m+1 \leq k < \infty}  \{ w_S(k/m) \}^{-1}  \frac{1}{m} \sum_{j=m}^{k-1}  \left|  \frac{k}{m} \{ W_2(m) + W_1(j-m) \}  -  \frac{j}{m}  \{ W_2(m) + W_1(k-m) \} \right|.
$$
Using Brownian scaling, the latter is equal in distribution to
\begin{align*}
  \sigma   \sup_{m+1 \leq k < \infty}&  \{ w_S(k/m) \}^{-1} \frac{1}{m} \sum_{j=m}^{k-1}  \left|  \frac{k}{m} \{ W_2(1) + W_1(j/m-1) \}  -  \frac{j}{m}  \{ W_2(1) + W_1(k/m-1) \} \right| \\
  = & \sigma   \sup_{m+1 \leq \ip{mt} < \infty}  \left\{ w_S\left(\frac{\ip{mt}}{m}\right) \right\}^{-1} \\ & \times \frac{1}{m} \sum_{j=m}^{\ip{mt}-1}  \left|  \frac{\ip{mt}}{m} \left\{ W_2(1) + W_1\left(\frac{j}{m}-1\right) \right\}  -  \frac{j}{m}  \left\{ W_2(1) + W_1\left(\frac{\ip{mt}}{m}-1\right) \right\} \right| \\
        =& \sigma \sup_{1 \leq t <\infty}  \left\{ w_\gamma\left(\frac{\ip{mt}}{m}\right) \right\}^{-1}  \left(\frac{\ip{mt}}{m}\right)^{-5/2 - \eta} \\ & \times \int_1^t \left|  \frac{\ip{mt}}{m} \left\{ W_2(1) + W_1\left(\frac{\ip{ms}}{m}-1\right) \right\}  -  \frac{\ip{ms}}{m}  \left\{ W_2(1) + W_1\left(\frac{\ip{mt}}{m}-1\right) \right\} \right| \dd s \\
        =& \sigma \sup_{1 \leq t <\infty}   \left\{ w_\gamma\left(\frac{\ip{mt}}{m}\right) \right\}^{-1} \left(\frac{\ip{mt}}{m}\right)^{-1} \int_1^t R_\eta\left(\frac{\ip{ms}}{m}, \frac{\ip{mt}}{m}\right) \dd s,
\end{align*}
where the second equality follows from the fact that the integrand is zero on the interval $[\ip{mt}/m, t]$ since $\ip{ms} = \ip{mt}$ for $s \in [\ip{mt}/m, t]$. Then, from~\eqref{eq:ineq:fg}, we obtain that
\begin{multline*}
\left| \sup_{1 \leq t <\infty}  \left\{ w_\gamma\left(\frac{\ip{mt}}{m}\right) \right\}^{-1} \left(\frac{\ip{mt}}{m}\right)^{-1} \int_1^t R_\eta\left(\frac{\ip{ms}}{m}, \frac{\ip{mt}}{m}\right) \dd s \right. \\ \left. - \sup_{1 \leq t <\infty}  \{ w_\gamma(t) \}^{-1} t^{-1} \int_1^t R_\eta(s,t) \dd s \right| \\ \leq \sup_{1 \leq t <\infty}\left| \left\{ w_\gamma\left(\frac{\ip{mt}}{m}\right) \right\}^{-1} \left(\frac{\ip{mt}}{m}\right)^{-1} \int_1^t R_\eta\left(\frac{\ip{ms}}{m}, \frac{\ip{mt}}{m}\right) \dd s \right. \\ \left. - \{ w_\gamma(t) \}^{-1} t^{-1} \int_1^t R_\eta(s,t) \dd s \right|,
\end{multline*}
which is smaller than $I_m + J_m$, where
\begin{align*}
  I_m &= \sup_{1 \leq t <\infty}\left\{ w_\gamma\left(\frac{\ip{mt}}{m}\right) \right\}^{-1} \left(\frac{\ip{mt}}{m}\right)^{-1} \left| \int_{1}^t R_\eta\left(\frac{\ip{ms}}{m}, \frac{\ip{mt}}{m}\right) \dd s -  \int_1^t R_\eta(s,t) \dd s \right|,  \\
  J_m &= \sup_{1 \leq t <\infty} \left| \left\{ w_\gamma\left(\frac{\ip{mt}}{m}\right) \right\}^{-1} \left(\frac{\ip{mt}}{m}\right)^{-1}  - \{ w_\gamma(t) \}^{-1} t^{-1} \right| \int_1^t R_\eta(s,t) \dd s.
\end{align*}
To show~\eqref{eq:wc:S}, we shall verify that both $I_m$ and $J_m$ converge to zero almost surely. For $I_m$, we have
\begin{align*}
I_m \leq \sup_{1 \leq t <\infty} \left\{ w_\gamma\left(\frac{\ip{mt}}{m}\right) \right\}^{-1} \left(\frac{\ip{mt}}{m}\right)^{-1} t \times \sup_{1 \leq s \leq t < \infty} \left| R_\eta\left(\frac{\ip{ms}}{m}, \frac{\ip{mt}}{m}\right)  -  R_\eta(s,t) \right|.
\end{align*}
Using~\eqref{eq:w:gamma} and the fact that $\sup_{1 \leq t <\infty} mt / \ip{mt} \leq 2$ as soon as $m \geq 2$, the first supremum on the right is bounded by $2 \epsilon^{-1}$. The second supremum converges to zero almost surely from the almost sure uniform continuity of $R_\eta$ on $\{(s,t)\in [1,\infty)^2: s\le t\}$ shown in Lemma~\ref{lem:unif:cont:R:eta}. For $J_m$, we have
\begin{align*}
  J_m \leq  \sup_{1 \leq t <\infty} \left| \left\{ w_\gamma\left(\frac{\ip{mt}}{m}\right) \right\}^{-1} \frac{mt}{\ip{mt}} - \{ w_\gamma(t) \}^{-1} \right|  \times \sup_{1 \leq s \leq t < \infty} R_\eta(s,t).
\end{align*}
Since the second supremum on the right is almost surely bounded by Lemma~\ref{lem:unif:cont:R:eta}, it suffices to verify that the first supremum converges to zero. The latter is smaller than
\begin{align*}
\sup_{1 \leq t <\infty} \left| \left\{ w_\gamma\left(\frac{\ip{mt}}{m}\right) \right\}^{-1} \frac{mt}{\ip{mt}} - \{ w_\gamma(t) \}^{-1}\frac{mt}{\ip{mt}} + \{ w_\gamma(t) \}^{-1}\frac{mt}{\ip{mt}}  - \{ w_\gamma(t) \}^{-1} \right| \\ \leq 2 \sup_{1 \leq t <\infty} \left| \left\{ w_\gamma\left(\frac{\ip{mt}}{m}\right) \right\}^{-1}  - \{ w_\gamma(t) \}^{-1} \right| + \epsilon^{-1} \sup_{1 \leq t <\infty} \frac{ mt-\ip{mt}}{\ip{mt}}.
\end{align*}
The first term on the right-hand side converges to zero by the uniform continuity of the function $t \mapsto 1/w_\gamma(t)$ on $[1,\infty)$. The second term converges to zero since it is smaller than $\epsilon^{-1} m^{-1}$.

It remains to prove~\eqref{eq:wc:T}. Notice first that the limit in~\eqref{eq:wc:T} is almost surely finite since
$$
\sup_{1 \leq t <\infty}  t^{-1/2} \sqrt{ \int_1^t \{ R_\eta(s,t) \}^2 \dd s} \leq \sup_{1 \leq s \leq t <\infty} R_\eta(s,t).
$$
Then, starting  from~\eqref{eq:tilde:Tm} and proceeding as previously, it can be verified that the random variable $\sup_{m+1 \leq k < \infty} \{ w_T(k/m) \}^{-1}  \tilde T_m(k)$ is equal in distribution to
\begin{multline*}
\sigma \sup_{1 \leq t <\infty}  \left\{ w_T\left(\frac{\ip{mt}}{m}\right) \right\}^{-1}  \Bigg( \int_1^t \Bigg[  \frac{\ip{mt}}{m} \left\{ W_2(1) + W_1\left(\frac{\ip{ms}}{m}-1\right) \right\}  \\ 
\phantom{\sigma \sup_{1 \leq t <\infty}  \left\{ w_T\left(\frac{\ip{mt}}{m}\right) \right\}^{-1}  \Bigg( }
 \frac{\ip{ms}}{m}  \left\{ W_2(1) + W_1\left(\frac{\ip{mt}}{m}-1\right) \right\} \Bigg]^2 \dd s \Bigg)^{1/2} \\ 
 = \sigma \sup_{1 \leq t <\infty}  \left\{ w_\gamma\left(\frac{\ip{mt}}{m}\right) \right\}^{-1} \left(\frac{\ip{mt}}{m}\right)^{-1/2} \sqrt{ \int_1^t \left\{ R_\eta\left(\frac{\ip{ms}}{m}, \frac{\ip{mt}}{m}\right) \right\}^2 \dd s }.
\end{multline*}
Then, from~\eqref{eq:ineq:fg},
\begin{multline*}
\left| \sup_{1 \leq t <\infty}  \left\{ w_\gamma\left(\frac{\ip{mt}}{m}\right) \right\}^{-1} \left(\frac{\ip{mt}}{m}\right)^{-1/2} \sqrt{ \int_1^t \left\{ R_\eta\left(\frac{\ip{ms}}{m}, \frac{\ip{mt}}{m}\right) \right\}^2 \dd s} \right. \\ \left. - \sup_{1 \leq t <\infty}  \{ w_\gamma(t) \}^{-1} t^{-1/2} \sqrt{\int_1^t \{ R_\eta(s,t) \}^2 \dd s} \right| \\ \leq \sup_{1 \leq t <\infty}\left| \left\{ w_\gamma\left(\frac{\ip{mt}}{m}\right) \right\}^{-1} \left(\frac{\ip{mt}}{m}\right)^{-1/2} \sqrt{ \int_1^t \left\{ R_\eta\left(\frac{\ip{ms}}{m}, \frac{\ip{mt}}{m}\right) \right\}^2 \dd s} \right. \\ \left. - \{ w_\gamma(t) \}^{-1} t^{-1/2} \sqrt{\int_1^t \{ R_\eta(s,t) \}^2 \dd s} \right|,
\end{multline*}
which is smaller than $I_m' + J_m'$, where
\begin{align*}
  I_m' =& \sup_{1 \leq t <\infty} \left\{ w_\gamma\left(\frac{\ip{mt}}{m}\right) \right\}^{-1} \left(\frac{\ip{mt}}{m}\right)^{-1/2} \\ & \times  \left| \sqrt{\int_1^t \left\{ R_\eta\left(\frac{\ip{ms}}{m}, \frac{\ip{mt}}{m}\right) \right\}^2 \dd s } - \sqrt{ \int_1^t \{ R_\eta(s,t) \}^2 \dd s}  \right|,  \\
  J_m' =& \sup_{1 \leq t <\infty} \left| \left\{ w_\gamma\left(\frac{\ip{mt}}{m}\right) \right\}^{-1} \left(\frac{\ip{mt}}{m}\right)^{-1/2}  - \{ w_\gamma(t) \}^{-1} t^{-1/2} \right| \sqrt{ \int_1^t \{ R_\eta(s,t) \}^2 \dd s } .
\end{align*}
Using Minkwoski's inequality in the form $|\|f\|_2-\|g\|_2|\le \|f-g\|_2$, and using similar arguments as previously,
\begin{align*}
  I_m' &\leq \sup_{1 \leq t <\infty} \left\{ w_\gamma\left(\frac{\ip{mt}}{m}\right) \right\}^{-1} \left(\frac{\ip{mt}}{m}\right)^{-1/2}  \sqrt{\int_1^t \left\{ R_\eta\left(\frac{\ip{ms}}{m}, \frac{\ip{mt}}{m}\right) - R_\eta(s,t) \right\}^2 \dd s} \\
  &\leq \epsilon^{-1} \sup_{1 \leq t <\infty} \left(\frac{mt}{\ip{mt}}\right)^{ 1/2} \times \sup_{1 \leq s \leq t < \infty} \left| R_\eta\left(\frac{\ip{ms}}{m}, \frac{\ip{mt}}{m}\right)  -  R_\eta(s,t) \right| \as 0.
\end{align*}
Proceeding as for $J_m$, $J_m' \as 0$ since
\begin{multline*}
\sup_{1 \leq t <\infty} \left| \left\{ w_\gamma\left(\frac{\ip{mt}}{m}\right) \right\}^{-1} \left(\frac{mt}{\ip{mt}}\right)^{1/2} - \{ w_\gamma(t) \}^{-1}  \right| \\ \leq \sqrt{2} \sup_{1 \leq t <\infty} \left| \left\{ w_\gamma\left(\frac{\ip{mt}}{m}\right) \right\}^{-1} - \{ w_\gamma(t) \}^{-1}  \right| + \epsilon^{-1}  \sup_{1 \leq t <\infty}  \frac{\sqrt{mt}-\sqrt{\ip{mt}}}{\sqrt{\ip{mt}}}  \as 0
\end{multline*}
by the uniform continuity of the functions $t \mapsto 1/w_\gamma(t)$ and the fact that the argument of the second supremum on the right hand side is at most $\ip{mt}^{-1}$.
\end{proof}

\begin{proof}[\bf Proof of Theorem~\ref{thm:H0}]
From Lemmas~\ref{lem:ae:tilde} and~\ref{lem:wc:tilde}, we have that, for any fixed $\eta > 0$, $\epsilon > 0$ and $\gamma \geq 0$,
\begin{align*}
\sup_{m+1 \leq k < \infty} \{ w_R(k/m) \}^{-1}  R_m(k) &\leadsto \sigma \sup_{1 \leq s \leq t <\infty} \{ w_\gamma(t) \}^{-1} R_\eta(s,t), \\
\sup_{m+1 \leq k < \infty} \{ w_S(k/m) \}^{-1}  S_m(k) &\leadsto \sigma \sup_{1 \leq t <\infty}  \{ w_\gamma(t) \}^{-1} t^{-1}  \int_1^t |  R_\eta(s,t) | \dd s, \\
\sup_{m+1 \leq k < \infty} \{ w_T(k/m) \}^{-1}  T_m(k) &\leadsto \sigma \sup_{1 \leq t <\infty} \{ w_\gamma(t) \}^{-1} t^{-1/2} \sqrt{ \int_1^t \{ R_\eta(s,t) \}^2 \dd s},
\end{align*}
where $w_\gamma$ is defined in~\eqref{eq:w:gamma} and the random function $R_\eta$ is defined in~\eqref{eq:R:eta}.

It thus remains to verify that the expressions of the limiting random variables can be simplified to coincide in distribution with those given in the statement of the theorem. The latter is merely a consequence of the fact that the random functions
$$
U(s,t) = (t-s) W_2(1) + t W_1(s-1) - s W_1(t-1), \qquad 1 \leq s \leq t,
$$
and
$$
V(s,t) = t W(s) - s W(t), \qquad 1 \leq s \leq t,
$$
are equal in distribution. Since $U$ and $V$ are centered Gaussian processes whose sample paths are continuous almost surely, the equality in distribution is a direct consequence of the equality of their covariance functions. Indeed, for any $1 \leq s \leq t$ and $1 \leq s' \leq t'$, it is an exercise to verify by direct computation that
\begin{align*}
\Cov\{U(s,t), U(s',t')\} &
= 
\Cov\{V(s,t), V(s',t')\}.
\end{align*}
\end{proof}

\begin{proof}[\bf Proof of Proposition~\ref{prop:R0:infinite}]
We have for all $k \geq 1$
\begin{align*}
  \sup_{1 \leq s \leq t < \infty} \frac{1}{t^{3/2}} | t W(s) - s W(t) | &\geq \frac{1}{(2^k)^{3/2}} | 2^k W(2^{k-1}) - 2^{k-1} W(2^k) | \\
                                           &= \frac{2^{k-1}}{2^{3k/2}} | W(2^{k-1}) -\{ W(2^k) -  W(2^{k-1}) \} | \\
                                           &= \frac{1}{2^{k/2+1}} | W(2^{k-1}) -\{ W(2^k) -  W(2^{k-1}) \} |.
\end{align*}
Consider an arbitrary fixed $M > 0$ and the events $D_k = \{ | W(2^{k-1}) -\{ W(2^k) -  W(2^{k-1}) \} | \geq M 2^{k/2+1} \}$. It is sufficient to show that $\Pr\left(
  \bigcup_{k=1}^\infty D_k \right) = 1$ (since this shows that the supremum is at least $M$ with probability 1), or equivalently, $\Pr\left(\bigcap_{k=1}^\infty D_k^\C \right) = 0$. Now,
\begin{align*}
\Pr\left(\bigcap_{k=1}^\infty D_k^\C \right) = \lim_{r \to \infty} \Pr\left(\bigcap_{k=1}^r D_k^\C \right) &= \lim_{r \to \infty} \Pr(D_1^\C) \prod_{k=2}^r \Pr \left(D_k^\C \,\Big |\, \bigcap_{j=1}^{k-1} D_j^\C \right) \\ &= \Pr(D_1^\C)\lim_{r \to \infty}  \prod_{k=2}^r \left\{ 1 - \Pr \left(D_k \,\Big |\,\bigcap_{j=1}^{k-1} D_j^\C \right) \right\},
\end{align*}
so that it is enough to show that there exists $\delta_M > 0$ such that $\Pr \left(D_k \mid \bigcap_{j=1}^{k-1} D_j^\C \right) \geq \delta_M$ for all $k \geq 2$. The latter holds with $\delta_M = \Pr(|Z| \geq 2^{3/2} M)/2 > 0$, where $Z$ is a standard normal random variable, since for all $k \geq 2$, with $A_k=\{ W(2^k) - W(2^{k-1}) \text{ and }W(2^{k-1}) \text{ have opposite sign}\}$,
\begin{align*}
\Pr \left(D_k \,\Big |\, \bigcap_{j=1}^{k-1} D_j^\C \right) & \geq \Pr \Big( A_k,| W(2^k) -  W(2^{k-1}) | \geq M 2^{k/2+1} \Big)\\
&=\Pr\Big(| W(2^k) -  W(2^{k-1}) | \geq M 2^{k/2+1} \Big)\Pr(A_k)\\
&= \Pr(|Z| \geq 2^{3/2} M)/2.
\end{align*}
In the above, we have used the fact that the sign of the increment and its magnitude are independent of the past and each other, and that the increment is Gaussian with mean $0$ and variance $2^{k-1}$.
\end{proof}


\section{Proofs of the results under alternatives}
\label{app:H1}

\begin{proof}[\bf Proof of Theorem~\ref{thm:H1:R}]
  Let us first prove the claim when $(iv)$ in Condition~\ref{cond:H1:R} holds. Recall the definition of the function $w_\gamma$ in~\eqref{eq:w:gamma} and notice that $w_\gamma(t) \leq 1$ for all $t \in [1,\infty)$, $\epsilon \geq 0$ and $\gamma \geq 0$.  Thus, for all $\eta \geq 0$, $\epsilon \geq 0$ and $\gamma \geq 0$,
  \begin{align}
    \nonumber
    \sup_{m+1 \leq k < \infty} & \{ w_R(k/m) \}^{-1}  R_m(k) \geq  \sup_{m+1 \leq k < \infty} (k/m)^{-3/2 - \eta}  R_m(k) \\
    \nonumber
    & \overset{k = k^\star_m + \ip{cm}}{\geq}  \left(\frac{k^\star_m + \ip{cm}}{m}\right)^{-3/2 - \eta}   \max_{m \leq j \leq k^\star_m + \ip{cm} - 1} \frac{j (k^\star_m + \ip{cm} - j)}{m^{3/2}}  \big| \bar X^{\sm}_{1:j} - \bar X^{\sm}_{j+1:k^\star_m + \ip{cm}} \big| \\
    \nonumber
    & \overset{j = k^\star_m}{\geq} \left(\frac{k^\star_m}{m} + c \right)^{-3/2 - \eta}  \frac{k^\star_m \ip{cm}}{m^{3/2}}  \big| \bar X^{\sm}_{1:k^\star_m} - \bar X^{\sm}_{k^\star_m+1:k^\star_m + \ip{cm}} \big| \\
    \nonumber
                               &  \geq \left(\frac{k^\star_m}{m} + c \right)^{-3/2 - \eta}  \frac{k^\star_m (cm - 1)}{m^{3/2}}  \big| \bar X^{\sm}_{1:k^\star_m} - \Ex(X^{\sm}_{k^\star_m}) + \Ex(X^{\sm}_{k^\star_m}) - \Ex(X^{\sm}_{k^\star_m+ 1}) \\
    \nonumber
                               & \phantom{\left(\frac{k^\star_m}{m} + c \right)^{-3/2 - \eta}  \frac{k^\star_m (cm - 1)}{m^{3/2}}  \big|}\qquad +  \Ex(X^{\sm}_{k^\star_m+ 1}) - \bar X^{\sm}_{k^\star_m+1:k^\star_m + \ip{cm}} \big| \\
    \nonumber
                               &  \geq m^{1/2} \left(\frac{k^\star_m}{m} + c \right)^{-3/2 - \eta}   \left( \frac{k^\star_m}{m} \right) \left(c - \frac{1}{m} \right)  \Big\{ | \Ex(X^{\sm}_{k^\star_m+1}) - \Ex(X^{\sm}_{k^\star_m}) |   \\
    \label{eq:last:eq}
                               & \qquad \qquad \qquad \qquad \qquad - | \bar X^{\sm}_{1:k^\star_m} - \Ex(X^{\sm}_{k^\star_m})- \bar X^{\sm}_{k^\star_m+1:k^\star_m + \ip{cm}} + \Ex(X^{\sm}_{k^\star_m+ 1}) | \Big\}
\end{align}
by the triangle inequality. Since $(iv)$ in Condition~\ref{cond:H1:R} holds, for $m$ large enough,~\eqref{eq:last:eq} is at least $(C_1 + c)^{-3/2 - \eta} \, c_1 \, (c/2)$ times 
\begin{equation}
  \label{eq:inf}
 m^{1/2} \big\{ | \Ex(Y^{\sm}_1) - \Ex(Y^{\sz}_1) |  - | \bar Y^{\sz}_{1:k^\star_m} - \Ex(Y^{\sz}_1) - \bar Y^{\sm}_{k^\star_m+1:k^\star_m + \ip{cm}} + \Ex(Y^{\sm}_1) | \big\}.
\end{equation}
Notice that~\eqref{eq:before} implies that $\sqrt{m} \{ \bar Y^{\sz}_{1:k^\star_m} - \Ex(Y^{\sz}_1) \} = (m/k^\star_m)^{1/2} \times \sqrt{k^\star_m} \{ \bar Y^{\sz}_{1:k^\star_m} - \Ex(Y^{\sz}_1) \} = O_\Pr(1)$ and that~\eqref{eq:after1} implies that $\sqrt{m} \{  \bar Y^{\sm}_{k^\star_m+1:k^\star_m + \ip{cm}} - \Ex(Y^{\sm}_1) \} = O_\Pr(1)$.  Hence,~\eqref{eq:inf} diverges in probability to infinity as a consequence of $(ii)$ in Condition~\ref{cond:H1:R}. 

Assume now that $(v)$ in Condition~\ref{cond:H1:R} holds. Then, we can use the fact that
\begin{align*}
  \sup_{m+1 \leq k < \infty}& \{ w_R(k/m) \}^{-1}  R_m(k) \overset{ k = k^\star_m + \ip{c k^\star_m}}{\geq}  \left(\frac{k^\star_m + \ip{c k^\star_m}}{m} \right)^{-3/2 - \eta}  R_m (k^\star_m + \ip{c k^\star_m}) \\
                            &\geq  \frac{m^{3/2 + \eta}}{(k^\star_m + c k^\star_m)^{3/2 + \eta}}  \max_{m \leq j \leq k^\star_m + \ip{c k^\star_m}-1} \frac{j (k^\star_m + \ip{c k^\star_m} - j)}{m^{3/2}}  \big| \bar X^{\sm}_{1:j} - \bar X^{\sm}_{j+1:k^\star_m + \ip{c k^\star_m}} \big| \\
  & \overset{ j = k^\star_m}{\geq}  \frac{m^\eta k^\star_m (c k^\star_m -1)}{(k^\star_m + c k^\star_m)^{3/2 + \eta}}  \big| \bar X^{\sm}_{1:k^\star_m} - \bar X^{\sm}_{k^\star_m+1:k^\star_m + \ip{c k^\star_m}} \big| \\
                            &\geq \frac{m^\eta (k^\star_m)^{1/2 - \eta} (c - 1/ k^\star_m)}{(1 + c)^{3/2 + \eta}}  \big| \bar X^{\sm}_{1:k^\star_m} - \Ex(X^{\sm}_{k^\star_m}) + \Ex(X^{\sm}_{k^\star_m}) \\
                            & \qquad \qquad \qquad \qquad \qquad \qquad  - \Ex(X^{\sm}_{k^\star_m+ 1}) +  \Ex(X^{\sm}_{k^\star_m+ 1}) - \bar X^{\sm}_{k^\star_m+1:k^\star_m + \ip{c k^\star_m}} \big|,
\end{align*}
which, for $m$ large enough, is larger than
\begin{equation}
  \label{eq:last:eq2}
  \frac{m^{1/2} c}{2 (1 + c)^{3/2 + \eta}}  \big\{ | \Ex(Y^{\sm}_1) - \Ex(Y^{\sz}_1) |   - | \bar Y^{\sz}_{1:k^\star_m} - \Ex(Y^{\sz}_1)  -  \bar Y^{\sm}_{k^\star_m+1:k^\star_m + \ip{c k^\star_m}} + \Ex(Y^{\sm}_1) | \big\}.
\end{equation}
This time, since $k^\star/m \to \infty$,~\eqref{eq:before} implies that $\sqrt{m} \{ \bar Y^{\sz}_{1:k^\star_m} - \Ex(Y^{\sz}_1) \} = (m/k^\star_m)^{1/2} \times \sqrt{k^\star_m} \{ \bar Y^{\sz}_{1:k^\star_m} - \Ex(Y^{\sz}_1) \} = o_\Pr(1)$ while~\eqref{eq:after2} implies that $(m/k^\star_m)^{1/2} \sqrt{k^\star_m} \{  \bar Y^{\sm}_{k^\star_m+1:k^\star_m + \ip{c k^\star_m}} - \Ex(Y^{\sm}_1) \} = o_\Pr(1)$. Therefore,~\eqref{eq:last:eq2} diverges in probability to infinity as a consequence of $(ii)$ in Condition~\ref{cond:H1:R}.
\end{proof}


\begin{proof}[\bf Proof of Theorem~\ref{thm:H1:ST}]
  We adapt the proof of Proposition~2.7 of \cite{KojVer20a} to the current setting. Given a set $\mathcal{S}$, let $\ell^\infty(\mathcal{S})$ denote the space of all bounded real-valued functions on $\mathcal{S}$ equipped with the uniform metric. Fix $T > c$. For any $s \in [1, T]$, let
$$
W_{m,Y}(s) = m^{-1/2} \sum_{i=1}^{\ip{ms}} \{ Y_i - \Ex(Y_1) \} \qquad \text{ and } \qquad W_{m,Z}(s) = m^{-1/2} \sum_{i=1}^{\ip{ms}} \{ Z_i - \Ex(Z_1) \}.
$$
From Condition~\ref{cond:H1:ST}, we have that $W_{m,Y}$ converges weakly to a standard Brownian motion $W_Y$ in $\ell^\infty([1,T])$ and $W_{m,Z}$ converges weakly to a standard Brownian motion $W_Z$ in $\ell^\infty([1,T])$.

Let $J_m(s,t) = m^{-1/2} H_m(s,t) - K_c(s,t)$, $(s,t) \in \Delta_T = \{(s,t) \in [1,T]^2 : s \leq t \}$. The fact that $m^{-1/2} H_m \p K_c$ in $\ell^\infty(\Delta_T)$ is proven if we show that
\begin{equation}
  \label{eq:claim1}
\sup_{(s,t) \in \Delta_T} | J_m(s,t) | \p 0.
\end{equation}
The supremum on the left-hand side of \eqref{eq:claim1} is equal to
\begin{equation}
\label{eq:max:sup}
\max \left\{\sup_{1 \leq s \leq t \leq c} | J_m(s,t) |,  \sup_{1 \leq s \leq c\leq t \leq T} | J_m(s,t) |, \sup_{c \leq s \leq t \leq T} | J_m(s,t) | \right\}.
\end{equation}
Notice first that
\begin{equation}
\label{eq:Kc:parts}
K_c(s,t) = \left\{
  \begin{array}{ll}
    0, &\text{if } 1 \leq s \leq t \leq c, \\
    s(t-c) \{\Ex(Y_1)- \Ex(Z_1)\}, &\text{if } 1 \leq s \leq c\leq t \leq T, \\
    c(t-s)   \{\Ex(Y_1)- \Ex(Z_1)\}, &\text{if } c \leq s \leq t \leq T.
  \end{array}
\right.
\end{equation}
Furthermore, for any $(s,t) \in \Delta_T \cap [1,c]^2$, let
$$
D_{m,Y}(s,t) = \sqrt{m} \lambda_m(s,t) \{ \bar X^{\sm}_{\ip{ms}+1:\ip{mt}} - \Ex(X^{\sm}_{k^\star_m}) \}  =  W_{m,Y}(t) - W_{m,Y}(s),
$$
where $\lambda_m(s,t) = (\ip{mt} - \ip{ms})/m$, $(s, t) \in \Delta_T$, and, for any $(s,t) \in \Delta_T \cap [c,T]^2$, let
$$
D_{m,Z}(s,t) = \sqrt{m} \lambda_m(s,t) \{ \bar X^{\sm}_{\ip{ms}+1:\ip{mt}} - \Ex(X^{\sm}_{k^\star_m+1}) \}  =  W_{m,Z}(t) - W_{m,Z}(s).
$$
Under the assumptions of the theorem, from the continuous mapping theorem, $D_{m,Y} \leadsto D_Y$ in $\ell^\infty (\Delta_T \cap [1,c]^2)$ and $D_{m,Z} \leadsto D_Z$ in $\ell^\infty(\Delta_T \cap [c,T]^2)$, where $D_Y(s,t) = W_Y(t) - W_Y(s)$ and $D_Z(s,t) = W_Z(t) - W_Z(s)$.

From the expression of $K_c$ given in~\eqref{eq:Kc:parts}, for the first supremum in~\eqref{eq:max:sup}, we obtain that
$$
\sup_{1 \leq s \leq t \leq c} | J_m(s,t) | = m^{-1/2} \sup_{1 \leq s \leq t \leq c} | H_m(s,t) | = o(1) \times O_\Pr(1) \p 0,
$$
since $H_m$ converges weakly to $(s,t) \mapsto (t-s) D_Y(0,s) - s D_Y(s,t)$ in $\ell^\infty(\Delta_T \cap [1,c]^2)$ as a consequence of the fact that, for any $1 \leq s \leq t \leq c$, $H_m(s,t) = \lambda_m(s,t) D_{m,Y}(0,s) -\lambda_m(0,s) D_{m,Y}(s,t)$ and $\sup_{(s,t) \in \Delta_T} |\lambda_m(s,t) - (t-s)| \leq 2/m$, and from the continuous mapping theorem.

Regarding the second supremum, for any $1 \leq s \leq c\leq t \leq T$, we have that
$$
\lambda_m(s,t) \bar X_{\ip{ms}+1:\ip{mt}} = \lambda_m(s,c) \bar X_{\ip{ms}+1:\ip{mc}} + \lambda_m(c,t) \bar X_{\ip{mc}+1:\ip{mt}}.
$$
Thus, on one hand,
$$
m^{-1/2} H_m(s,t) = \lambda_m(0,s) \{ \lambda_m(s,t) \bar X_{1:\ip{ms}} - \lambda_m(s,c) \bar X_{\ip{ms}+1:\ip{mc}} - \lambda_m(c,t) \bar X_{\ip{mc}+1:\ip{mt}} \}.
$$
On the other hand, from~\eqref{eq:Kc:parts} and using again the fact that $\sup_{(s,t) \in \Delta_T} |\lambda_m(s,t) - (t-s)| \leq 2/m$,
$$
K_c(s,t) =  \lambda_m(0,s) \{ \lambda_m(s,t)  \Ex(Y_1) - \lambda_m(s,c) \Ex(Y_1) - \lambda_m(c,t) \Ex(Z_1) \} + O(1/m),
$$
where the term $O(1/m)$ is uniform in $s,t,c$. By the triangle inequality and using the fact that $\sup_{(s,t) \in \Delta_T} |\lambda_m(s,t)| \leq T$, it then follows that
\begin{multline*}
  \sup_{1 \leq s \leq c\leq t \leq T} | J_m(s,t) | \leq m^{-1/2} \, T \left[ \sup_{1 \leq s \leq c} |D_{m,Y}(0,s)|  \right. \\ \left. + \sup_{1 \leq s \leq c} |D_{m,Y}(s,c)| + \sup_{c\leq t \leq T} | D_{m,Z}(c,t) |  \right] = o(1) \times O_\Pr(1).
\end{multline*}
Similarly, for the third supremum, for any $c \leq s \leq T$,
$$
\lambda_m(0,s) \bar X_{1:\ip{ms}} = \lambda_m(0,c) \bar X_{1:\ip{mc}} + \lambda_m(c,s) \bar X_{\ip{mc}+1:\ip{ms}},
$$
and, hence, on one hand, for any $c \leq s \leq t \leq T$,
\begin{multline*}
m^{-1/2} H_m(s,t) = \lambda_m(s,t) \{ \lambda_m(0,c) \bar X_{1:\ip{mc}} + \lambda_m(c,s) \bar X_{\ip{mc}+1:\ip{ms}} - \lambda_m(0,s) \bar X_{\ip{ms}+1:\ip{mt}}\},
\end{multline*}
while, on the other hand,
$$
K_c(s,t) =  \lambda_m(s,t) \{ \lambda_m(0,c)  \Ex(Y_1) +  \lambda_m(c,s) \Ex(Z_1) - \lambda_m(0,s) \Ex(Z_1) \} + O(1/m),
$$
with the term $O(1/m)$ again uniform in $s,t,c$. Finally, by the triangle inequality,
\begin{multline*}
  \sup_{c \leq s \leq t \leq T} | J_m(s,t) |  \leq m^{-1/2} T \left[ | D_{m,Y}(0,c) | \right. \\ \left. + \sup_{c\leq s \leq T} | D_{m,Z}(c,s) | + \sup_{c\leq s \leq t \leq T}  | D_{m,Z}(s,t) |  \right] = o(1) \times O_\Pr(1),
\end{multline*}
which completes the proof of~\eqref{eq:claim1}.

It remains to prove~\eqref{eq:H1:ST}. Recall that from the definition of the function $w_\gamma$ in~\eqref{eq:w:gamma}, $w_\gamma(t) \leq 1$ for all $t \geq 1$, $\epsilon > 0$ and $\gamma \geq 0$. Hence,
\begin{align*}
  \sup_{m+1 \leq k < \infty} &\{ w_S(k/m) \}^{-1}  S_m(k) \geq \sup_{m+1 \leq k \leq \ip{mT}} (k/m)^{-5/2-\eta}  \frac{1}{m} \sum_{j=m}^{k-1} \frac{j (k-j)}{m^{3/2}}  | \bar X^{\sm}_{1:j} - \bar X^{\sm}_{j+1:k} | \\
                             &= \sup_{t \in [1,T]} \left( \frac{\ip{mt}}{m} \right)^{-5/2-\eta}  \int_1^t |H_m(s,t)| \dd s \geq \sup_{t \in [1,T]} t^{-5/2-\eta}  \int_1^t |H_m(s,t)| \dd s \\
                             &\geq T^{-5/2-\eta} \, m^{1/2} \sup_{t \in [1,T]} \int_1^t |m^{-1/2} H_m(s,t)| \dd s \p \infty
\end{align*}
since, by the continuous mapping theorem,
$$
\sup_{t \in [1,T]} \int_1^t |m^{-1/2} H_m(s,t)| \dd s \p \sup_{t \in [1,T]} \int_1^t |K_c(s,t)| \dd s > 0.
$$
Similarly, $\sup_{m+1 \leq k < \infty} \{ w_T(k/m) \}^{-1}  T_m(k)$ is larger than
\begin{align*}
  \sup_{m+1 \leq k \leq \ip{mT}} &(k/m)^{-2-\eta}  \sqrt{ \frac{1}{m}  \sum_{j=m}^{k-1} \left\{ \frac{j (k-j)}{m^{3/2}}  ( \bar X_{1:j} - \bar X_{j+1:k} ) \right\}^2 } \\
                             &\geq \sup_{t \in [1,T]} t^{-2-\eta}  \sqrt{\int_1^t \{ H_m(s,t) \}^2 \dd s}  \\
                             &\geq T^{-2-\eta} m^{1/2} \sup_{t \in [1,T]} \sqrt{ \int_1^t \{ m^{-1/2} H_m(s,t) \}^2 \dd s } \p \infty
\end{align*}
since, by the continuous mapping theorem,
$$
\sup_{t \in [1,T]} \sqrt{ \int_1^t \{ m^{-1/2} H_m(s,t) \}^2 \dd s } \p \sup_{t \in [1,T]} \sqrt{ \int_1^t \{ K_c(s,t) \}^2 \dd s }  > 0.
$$
\end{proof}

\def\propequivdistR{\ref{prop:equiv:dist:R}}
\def\propthetaIF{\ref{prop:theta:IF}}
\section{Proofs of Propositions~\propequivdistR~and~\propthetaIF}
\label{app:others}

\begin{proof}[\bf Proof of Proposition~\ref{prop:equiv:dist:R}]
For any fixed $\eta > 0$, $\epsilon > 0$ and $\gamma \geq 0$, we have
\begin{align*}
  \sup_{1 \leq s \leq t < \infty} &\frac{1}{t^{3/2+\eta} \max[ \{ (t-1)/t \}^\gamma, \epsilon]  } |t W(s) - s W(t) | \\
                                  &=\sup_{1 \leq s \leq t < \infty}\frac{1}{t^{3/2+\eta}\max\{(1-1/t)^\gamma,\epsilon\}} \left| \frac{ts}{s} W(s) - \frac{st}{t} W(t) \right|\\
                                  &=\sup_{1 \leq s \leq t < \infty}\frac{ts}{t^{3/2+\eta}\max\{(1-1/t)^\gamma,\epsilon\}} \left| \frac{1}{s} W(s) - \frac{1}{t} W(t) \right| \\
                                  &=\sup_{1 \leq s \leq t < \infty}\frac{s}{t^{1/2+\eta}\max\{(1-1/t)^\gamma,\epsilon\}} \left| \frac{1}{s} W(s) - \frac{1}{t} W(t) \right|.
\end{align*}
Let $u=1/t$ and $v=1/s$. Then, the last expression on the right is equal to
\begin{multline*}
\sup_{1\le   1/v \leq 1/u < \infty}\frac{1/v}{(1/u)^{1/2+\eta}\max\{(1-u)^\gamma,\epsilon\}}|v W(1/v) - u W(1/u) |\\
=\sup_{0<u\le v \le 1}\frac{u^{1/2+\eta}}{v\max\{(1-u)^\gamma,\epsilon\}}|v W(1/v) - u W(1/u) |.
\end{multline*}
The claim finally follows from the Brownian inversion property stating that $\{ t W(1/t) \}_{t\geq 0}$ is also a standard Brownian motion.
\end{proof}

\begin{proof}[\bf Proof of Proposition~\ref{prop:theta:IF}]
Fix $\eta > 0$, $\epsilon > 0$ and $\gamma \geq 0$, and, for any $k \geq m + 1$, let
\begin{align}
  \nonumber
  V_m(k) &= \epsilon^{-1} \sup_{m+1 \leq k \leq \infty}  (k/m)^{-3/2-\eta} \max_{m \leq j \leq k-1} \frac{j (k-j)}{m^{3/2}} | \theta_{1:j} - \theta_{j+1:k} - \bar \IF_{1:j} + \bar \IF_{j+1:k} | \\
  \label{eq:Vm}
         &= \epsilon^{-1} \sup_{m+1 \leq k \leq \infty}  (k/m)^{-3/2-\eta} \max_{m \leq j \leq k-1} \frac{j (k-j)}{m^{3/2}} | R_{1,j} - R_{j+1:k}  |.
\end{align}
Then, using the reverse triangle inequality for the maximum norm,
\begin{multline*}
\sup_{m+1 \leq k \leq \infty}  \{ w_R(k/m) \}^{-1} | R_m^\theta(k) - R_m^\IF(k) | \leq V_m(k) \\
  \leq \epsilon^{-1} m^\eta \sup_{m+1 \leq k \leq \infty}  k^{-3/2-\eta}  \left\{ \max_{m \leq j \leq k-1} j (k-j) | R_{1,j} |  +   \max_{m \leq j \leq k-1} j (k-j) | R_{j+1,k} | \right\} = o_\Pr(1)
\end{multline*}
since
\begin{align*}
  m^\eta & \sup_{m+1 \leq k \leq \infty}  k^{-3/2-\eta}  \max_{m \leq j \leq k-1} j (k-j) | R_{1,j} |  \leq  m^\eta \sup_{m+1 \leq k \leq \infty}  k^{-1/2-\eta}  \max_{m \leq j \leq k-1} j | R_{1,j} | \\
         & \leq  m^\eta \sup_{m+1 \leq k \leq \infty}  k^{-\eta}   \times \sup_{m+1 \leq k \leq \infty} k^{-1/2} \max_{m \leq j \leq k-1} j | R_{1,j} | \\
         & \leq \sup_{m+1 \leq k \leq \infty} k^{-1/2} \max_{1 \leq i < j \leq k} (j - i + 1) | R_{i,j} | = o_\Pr(1)
\end{align*}
as a consequence of Condition~\ref{cond:remainders} and, similarly,
\begin{align*}
  m^\eta & \sup_{m+1 \leq k \leq \infty}  k^{-3/2-\eta}  \max_{m \leq j \leq k-1} j (k-j) | R_{j+1,k} | \leq  m^\eta \sup_{m+1 \leq k \leq \infty}  k^{-1/2-\eta}  \max_{m \leq j \leq k-1} (k-j) | R_{j+1,k} | \\
         &\leq  m^\eta \sup_{m+1 \leq k \leq \infty}  k^{-\eta}  \times \sup_{m+1 \leq k \leq \infty} k^{-1/2} \max_{m \leq j \leq k-1} (k-j) | R_{j+1,k} | \\
         & \leq \sup_{m+1 \leq k \leq \infty} k^{-1/2} \max_{1 \leq i < j \leq k} (j - i + 1) | R_{i,j} | = o_\Pr(1).
\end{align*}
Using the reverse triangle inequality for the $L_1$ norm, the claim for $S_m^\theta$ and $S_m^\IF$ follows from the fact that
\begin{multline*}
\sup_{m+1 \leq k \leq \infty}  \{ w_S(k/m) \}^{-1} | S_m^\theta(k) - S_m^\IF(k) |  \\ \leq \epsilon^{-1} \sup_{m+1 \leq k \leq \infty}  (k/m)^{-5/2-\eta} \frac{1}{m} \sum_{j=m}^{k-1} \frac{j (k-j)}{m^{3/2}} | R_{1,j} - R_{j+1:k} | \leq \frac{k-m}{k} V_m(k),
\end{multline*}
 where $V_m$ is defined in~\eqref{eq:Vm}. Similarly, the claim for $T_m^\theta$ and $T_m^\IF$ follows from the fact that
\begin{multline*}
\sup_{m+1 \leq k \leq \infty}  \{ w_T(k/m) \}^{-1} | T_m^\theta(k) - T_m^\IF(k) |  \\ \leq \epsilon^{-1} \sup_{m+1 \leq k \leq \infty}  (k/m)^{-2-\eta} \sqrt{\frac{1}{m} \sum_{j=m}^{k-1} \left\{\frac{j (k-j)}{m^{3/2}}  (R_{1,j} - R_{j+1:k}) \right\}^2 } \leq \sqrt{\frac{k-m}{k}} V_m(k).
\end{multline*}
\end{proof}

\end{appendix}

\bibliographystyle{imsart-nameyear}
\bibliography{biblio}

\end{document}

%% file: quantiles_RST.tex
\begin{table}[t!]
\centering
\caption{For $\alpha \in \{0.01, 0.05, 0.1\}$, estimated $(1-\alpha)$-quantiles of the limiting distributions appearing in Theorem~\ref{thm:H0} related to the monitoring schemes based on $R_m$, $S_m$ and $T_m$ for $\eta = 0.001$ and different values of $\gamma$. The quantiles were estimated using asymptotic regression models. Standard errors of the estimates are provided between parentheses. Estimated quantiles for larger values of $\eta$ are available in the \textsf{R} package \texttt{npcp}.} 
\label{tab:quantiles:RST}
\begingroup\footnotesize
\begin{tabular}{lrrrrrr}
  \hline
   & \multicolumn{2}{c}{$R_m$} & \multicolumn{2}{c}{$S_m$} & \multicolumn{2}{c}{$T_m$} \\ \cmidrule(lr){2-3} \cmidrule(lr){4-5} \cmidrule(lr){6-7} $1-\alpha$ & \multicolumn{1}{c}{$\gamma = 0$} & \multicolumn{1}{c}{$\gamma = 0.25$} & \multicolumn{1}{c}{$\gamma = 0$} & \multicolumn{1}{c}{$\gamma = 0.85$} & \multicolumn{1}{c}{$\gamma = 0$} & \multicolumn{1}{c}{$\gamma = 0.45$} \\ \hline
0.99 & 2.157(0.006) & 2.278(0.002) & 1.145(0.019) & 1.199(0.008) & 1.246(0.006) & 1.324(0.007) \\ 
  0.95 & 1.956(0.009) & 2.054(0.003) & 1.007(0.015) & 1.058(0.007) & 1.121(0.011) & 1.164(0.005) \\ 
  0.90 & 1.837(0.008) & 1.952(0.007) & 0.939(0.014) & 0.987(0.006) & 1.046(0.010) & 1.087(0.004) \\ 
   \hline
\end{tabular}
\endgroup
\end{table}

%% file: H0.tex
\begin{table}[t!]
\centering
\caption{Percentages of rejection of $H_0$ in~\eqref{eq:H0} for the procedures based on $R_m$, $S_m$, $T_m$ with $\gamma = 0$ and $\eta \in \{0.1, 0.05, 0.01, 0.001\}$, as well as for the procedures based on $E_m$ and $Q_m$ with $\gamma = 0$. The rejection percentages are computed from 5000 samples of size $n = m + 10000$ generated from the time series models M1 -- M10.} 
\label{tab:H0}
\begingroup\scriptsize
\begin{tabular}{lrrrrrrrrrrrrrrr}
  \hline
  \multicolumn{2}{c}{} & \multicolumn{4}{c}{$R_m$ with $\eta =$} & \multicolumn{4}{c}{$S_m$ with $\eta =$} & \multicolumn{4}{c}{$T_m$ with $\eta =$}  \\ \cmidrule(lr){3-6} \cmidrule(lr){7-10} \cmidrule(lr){11-14}  Model & $m$ & 0.05 & 0.01 & 0.005 & 0.001  & 0.05 & 0.01 & 0.005 & 0.001  & 0.05 & 0.01 & 0.005 & 0.001 & $E_m$ & $Q_m$\\ \hline
M1 & 100 & 8.7 & 7.8 & 7.4 & 7.1 & 5.5 & 4.8 & 4.6 & 4.5 & 6.9 & 5.4 & 5.4 & 5.4 & 6.1 & 6.7 \\ 
   & 200 & 6.7 & 5.2 & 4.7 & 4.5 & 3.7 & 2.8 & 2.8 & 2.8 & 4.6 & 3.7 & 3.6 & 3.5 & 4.9 & 5.1 \\ 
   & 400 & 5.4 & 3.6 & 3.1 & 2.8 & 2.4 & 1.6 & 1.4 & 1.2 & 3.3 & 2.3 & 2.1 & 2.0 & 4.7 & 5.0 \\ 
   & 800 & 4.3 & 2.5 & 2.1 & 1.7 & 1.7 & 0.9 & 0.7 & 0.7 & 2.6 & 1.4 & 1.2 & 1.1 & 3.1 & 3.9 \\ 
  M2 & 100 & 9.0 & 8.2 & 7.8 & 7.6 & 6.0 & 5.1 & 4.9 & 4.8 & 7.2 & 6.1 & 5.9 & 6.0 & 6.3 & 6.6 \\ 
   & 200 & 6.9 & 5.3 & 5.0 & 4.7 & 3.8 & 3.0 & 2.9 & 2.8 & 4.9 & 3.8 & 3.7 & 3.6 & 4.8 & 5.2 \\ 
   & 400 & 5.4 & 3.6 & 3.2 & 2.7 & 2.4 & 1.6 & 1.4 & 1.3 & 3.5 & 2.3 & 2.2 & 2.1 & 4.8 & 5.1 \\ 
   & 800 & 4.2 & 2.5 & 2.1 & 1.9 & 1.7 & 0.9 & 0.7 & 0.7 & 2.6 & 1.4 & 1.2 & 1.1 & 3.2 & 3.9 \\ 
  M3 & 100 & 10.4 & 9.8 & 9.5 & 9.3 & 7.1 & 6.2 & 6.0 & 5.9 & 8.4 & 7.4 & 7.4 & 7.3 & 6.7 & 7.2 \\ 
   & 200 & 7.5 & 5.9 & 5.5 & 5.3 & 4.1 & 3.4 & 3.2 & 3.1 & 5.4 & 4.2 & 4.1 & 4.0 & 4.9 & 5.6 \\ 
   & 400 & 5.5 & 3.6 & 3.2 & 2.9 & 2.6 & 1.6 & 1.5 & 1.4 & 3.7 & 2.5 & 2.4 & 2.3 & 4.9 & 5.2 \\ 
   & 800 & 4.2 & 2.5 & 2.1 & 1.9 & 1.8 & 0.9 & 0.9 & 0.7 & 2.7 & 1.5 & 1.2 & 1.2 & 3.3 & 4.0 \\ 
  M4 & 100 & 13.1 & 12.3 & 11.9 & 11.6 & 9.0 & 8.2 & 8.0 & 7.9 & 10.7 & 9.7 & 9.6 & 9.6 & 7.5 & 7.8 \\ 
   & 200 & 8.9 & 7.3 & 6.8 & 6.4 & 4.8 & 4.0 & 3.8 & 3.6 & 6.3 & 5.0 & 4.8 & 4.7 & 5.3 & 6.0 \\ 
   & 400 & 6.1 & 4.0 & 3.6 & 3.3 & 2.8 & 1.8 & 1.7 & 1.6 & 4.0 & 2.7 & 2.5 & 2.5 & 5.0 & 5.4 \\ 
   & 800 & 4.4 & 2.4 & 2.1 & 1.9 & 1.9 & 1.0 & 0.9 & 0.8 & 2.8 & 1.7 & 1.5 & 1.4 & 3.3 & 4.0 \\ 
  M5 & 100 & 18.4 & 18.3 & 17.7 & 17.4 & 13.9 & 13.1 & 12.8 & 12.6 & 15.5 & 14.8 & 14.5 & 14.6 & 9.6 & 9.7 \\ 
   & 200 & 11.5 & 10.1 & 9.6 & 9.4 & 7.3 & 6.3 & 5.9 & 5.6 & 8.9 & 7.2 & 7.0 & 6.9 & 6.6 & 7.2 \\ 
   & 400 & 7.6 & 5.2 & 4.7 & 4.4 & 3.8 & 2.6 & 2.5 & 2.2 & 5.0 & 3.5 & 3.4 & 3.3 & 5.3 & 6.0 \\ 
   & 800 & 4.9 & 2.9 & 2.5 & 2.2 & 2.1 & 1.3 & 1.1 & 1.1 & 3.1 & 2.0 & 1.8 & 1.7 & 3.7 & 4.4 \\ 
  M6 & 100 & 10.8 & 10.0 & 9.4 & 8.9 & 6.8 & 5.9 & 5.7 & 5.6 & 8.2 & 6.9 & 6.8 & 6.8 & 6.8 & 6.9 \\ 
   & 200 & 8.3 & 6.4 & 5.8 & 5.4 & 4.3 & 3.3 & 3.1 & 2.9 & 5.7 & 4.4 & 4.2 & 4.1 & 5.2 & 5.6 \\ 
   & 400 & 5.9 & 3.8 & 3.4 & 3.1 & 2.9 & 1.8 & 1.7 & 1.6 & 3.7 & 2.7 & 2.4 & 2.3 & 4.7 & 5.2 \\ 
   & 800 & 4.5 & 2.4 & 2.1 & 1.9 & 1.9 & 1.0 & 1.0 & 0.9 & 2.7 & 1.6 & 1.4 & 1.4 & 3.4 & 3.8 \\ 
  M7 & 100 & 12.8 & 11.4 & 10.8 & 10.5 & 7.7 & 6.9 & 6.8 & 6.5 & 9.3 & 8.3 & 8.1 & 7.9 & 6.7 & 6.8 \\ 
   & 200 & 9.6 & 7.6 & 7.1 & 6.7 & 4.8 & 4.1 & 3.9 & 3.7 & 6.1 & 5.0 & 4.8 & 4.8 & 5.3 & 5.5 \\ 
   & 400 & 6.4 & 4.2 & 3.6 & 3.4 & 2.9 & 1.8 & 1.5 & 1.5 & 3.8 & 2.6 & 2.5 & 2.5 & 4.4 & 4.8 \\ 
   & 800 & 5.0 & 3.0 & 2.4 & 2.1 & 1.8 & 1.2 & 1.1 & 1.0 & 2.9 & 1.6 & 1.5 & 1.4 & 3.4 & 3.9 \\ 
  M8 & 100 & 9.3 & 8.5 & 8.0 & 7.8 & 6.7 & 5.9 & 5.7 & 5.6 & 7.9 & 6.8 & 6.7 & 6.6 & 6.6 & 7.1 \\ 
   & 200 & 6.9 & 5.6 & 5.2 & 4.9 & 4.1 & 3.4 & 3.2 & 3.1 & 5.1 & 3.9 & 3.8 & 3.7 & 5.1 & 5.5 \\ 
   & 400 & 5.4 & 3.5 & 3.1 & 2.8 & 2.5 & 1.5 & 1.4 & 1.3 & 3.9 & 2.4 & 2.1 & 2.1 & 4.8 & 5.1 \\ 
   & 800 & 4.2 & 2.4 & 2.2 & 1.9 & 1.8 & 1.0 & 0.9 & 0.8 & 2.8 & 1.5 & 1.4 & 1.3 & 3.7 & 4.1 \\ 
  M9 & 100 & 36.0 & 35.7 & 35.1 & 34.8 & 28.0 & 26.9 & 26.6 & 26.1 & 31.1 & 29.9 & 29.6 & 29.6 & 15.2 & 13.0 \\ 
   & 200 & 28.8 & 26.3 & 25.3 & 24.7 & 18.6 & 15.7 & 15.2 & 15.0 & 22.6 & 19.6 & 19.1 & 18.8 & 11.7 & 10.4 \\ 
   & 400 & 20.5 & 17.0 & 16.0 & 15.3 & 11.0 & 8.6 & 8.2 & 7.8 & 14.4 & 11.5 & 11.0 & 10.7 & 9.6 & 8.7 \\ 
   & 800 & 14.8 & 10.3 & 9.4 & 8.8 & 5.5 & 3.7 & 3.4 & 3.1 & 9.0 & 5.8 & 5.4 & 5.1 & 7.5 & 7.2 \\ 
  M10 & 100 & 8.9 & 7.9 & 7.4 & 7.1 & 5.5 & 5.0 & 4.8 & 4.8 & 6.6 & 5.6 & 5.6 & 5.6 & 6.4 & 6.5 \\ 
   & 200 & 7.8 & 5.6 & 5.1 & 4.9 & 4.2 & 3.2 & 3.1 & 3.0 & 5.1 & 4.3 & 4.2 & 4.0 & 5.0 & 5.4 \\ 
   & 400 & 5.0 & 3.3 & 2.9 & 2.6 & 2.4 & 1.5 & 1.4 & 1.3 & 3.5 & 2.2 & 2.0 & 2.0 & 5.0 & 5.2 \\ 
   & 800 & 4.1 & 2.3 & 2.1 & 1.9 & 1.5 & 0.9 & 0.8 & 0.7 & 2.4 & 1.3 & 1.2 & 1.1 & 3.5 & 3.6 \\ 
   \hline
\end{tabular}
\endgroup
\end{table}

%% file: H1.tex
\begin{table}[t!]
\centering
\caption{Percentages of rejection of $H_0$ in~\eqref{eq:H0} for the procedures based on $R_m$, $S_m$, $T_m$ with $\gamma = 0$ and $\eta \in \{0.1, 0.05, 0.01, 0.005, 0.001\}$ estimated from 2000 samples of size $n = 20000$ from model M1 with $m=100$ such that, for each sample, a positive offset of $\delta=0.1$ was added to all observations after position $k^\star=15000$. The corresponding rejection percentages of the procedures based on $E_m$ and $Q_m$ are 0.7 and 0.7, respectively.} 
\label{tab:H1}
\begingroup\footnotesize
\begin{tabular}{rrrrrr}
  \hline
  & $\eta = 0.1$ & $\eta = 0.05$ & $\eta = 0.01$ & $\eta = 0.005$ & $\eta = 0.001$\\ \hline
$R_m$ & 62.3 & 88.9 & 95.5 & 95.6 & 95.7 \\ 
  $S_m$ & 64.4 & 79.7 & 87.1 & 87.6 & 87.7 \\ 
  $T_m$ & 66.3 & 84.9 & 92.1 & 92.5 & 92.8 \\ 
   \hline
\end{tabular}
\endgroup
\end{table}